\documentclass[11pt]{article}

\usepackage[a4paper,margin=30mm]{geometry}
\usepackage{amsmath,amssymb,amsthm,mathtools}
\usepackage[hidelinks,bookmarksdepth=2]{hyperref}
\newtheorem{theorem}{Theorem}[section]
\newtheorem{proposition}[theorem]{Proposition}
\newtheorem{lemma}[theorem]{Lemma}
\newtheorem{corollary}[theorem]{Corollary}
\theoremstyle{plain}
\newtheorem{definition}[theorem]{Definition}

\newtheorem{example}[theorem]{Example}
\theoremstyle{remark}
\newtheorem{remark}[theorem]{Remark}

\newcommand{\Rplus}{(0,\infty)}
\newcommand{\G}{G}
\newcommand{\C}{\mathcal C}
\newcommand{\1}{\mathbb I}
\newcommand{\E}{\mathbb E}
\newcommand{\Pp}{\mathbb P}
\newcommand{\Law}{\mathcal L}
\newcommand{\RV}{\mathrm{RV}}
\newcommand{\od}{\mathbin\odot}

\title{Multiradial Regular Variation}

\author{Enkelejd Hashorva\thanks{Department of Actuarial Science,
University of Lausanne, UNIL-Dorigny, 1015 Lausanne, Switzerland.
Email: \texttt{Enkelejd.Hashorva@unil.ch}}}

\date{September 20, 2026}

\begin{document}
\maketitle
   \begin{abstract}
   We introduce and study multiradial regular variation of random fields,
   allowing componentwise thresholds to diverge at unrelated rates. The
   limit measures are homogeneous in each component and finite on events
   where every component exceeds a positive level in absolute value
   somewhere on a compact window. We characterise convergence by anchor
   exceedance masses and conditional whole-path laws, together with a
   compact-window mass bound in the continuous-parameter setting. For
   fields indexed by a countable discrete abelian group or by \(\mathbb R^m\),
   \(m\ge1\),
   every shift-invariant tail measure in this class admits a strictly
   stationary realisation. Finite moving averages with shared volatility
   show that scalar row-tail measures and simultaneous-exceedance tail
   masses on every
   consecutive finite window can coincide while relative-lag tail masses
   differ.
\end{abstract}


\section{Introduction}
\label{1}

Regular variation describes extreme configurations through a limiting
measure. For a random vector $Y$, we use the direct-threshold
formulation
\begin{equation}
 \lim_{n\to\infty}a_n\Pp\{n^{-1}Y\in A\}=\nu(A),
 \qquad \nu(cA)=c^{-\alpha}\nu(A),\quad c>0,
 \label{2}
\end{equation}
where $\alpha>0$, $(a_n)_{n\ge1}$ is a positive regularly varying sequence
of index $\alpha$, the measure $\nu$ is nonzero and finite on sets separated
from the origin, and $A$ ranges over such $\nu$-continuity sets
\cite{basrak2002characterization,Resnick2007}.
 
For independent
tail-balanced components with a common tail index, the ordinary
multivariate tail measure is concentrated on the coordinate axes.
Simultaneous large values occur on a smaller probability scale,
motivating different normalisations and localisation, as in hidden
regular variation \cite{LindskogResnickRoy2014}.

In this contribution we study the joint tail layer in which every component is large,
possibly at a different location, and the component thresholds diverge
at unrelated rates. The underlying action is componentwise positive
scaling.\\ 
 We have studied multihomogeneous measures and their stochastic polar
representations  in \cite{Hashorva2026Multihomogeneous,Hashorva2026Particles};
see \cite{EvansMolchanov2018,BladtHashorvaShevchenko2022,hashorva2026shift} for the scalar action. Related joint-tail
approximations with separately diverging thresholds appear in
\cite[Theorem~3.1]{KonstantinidesPassalidis2025}. Here we formulate
joint-tail asymptotics as convergence of measures on path space,
with a common limit for all diverging component thresholds.

Our localisation builds on measure convergence on general spaces
\cite{deHaanLin2001,HultLindskog2006Skorokhod,hult:lindskog:2006},
tail and spectral processes
\cite{BasrakSegers2009,meinguet2010regularly,segers2017polar,DombryHashorvaSoulier2018,Hrovje,mikosch2024extreme}
and compact-window formulations
\cite{Soulier2022,BladtHashorvaShevchenko2022}. Following the boundedness
approach of \cite{basrak2019note,KulikSoulier2020}, we specify the sets
on which measures must be finite. A general topological framework for
$q=1$ is developed in \cite{BasrakMilincevicMolchanov2025}.\\ 
The boundedness $\mathcal B_\Pi$, defined in \eqref{182}, requires
each component to exceed a positive level in absolute value somewhere
on a compact window. These locations may differ across components.

For a $q$-dimensional random field (rf) $X=(X_1,\ldots,X_q)$ observed at
$t_1,\ldots,t_k$, the underlying object is a random matrix: rows
correspond to components and columns to observation locations.
Each row has its own threshold:
\[
 \mathbf X=(X_i(t_j))_{1\le i\le q,\,1\le j\le k},
 \qquad
 u^{-1}\od\mathbf X
 =\operatorname{diag}(u_1^{-1},\ldots,u_q^{-1})\mathbf X.
\]
The relevant sets keep every row norm away from zero, although
individual entries may vanish and different rows may be visible in
different columns. We work with finite or countable discrete index
sets and with continuous paths on Euclidean parameter spaces, equipped
with locally uniform convergence.

For example, requiring each component to exceed its level in absolute
value somewhere in a window differs from requiring a joint exceedance
at one common location. The former allows displacement between
component extremes; the latter requires their overlap. Both are
localised joint-tail events, but their limiting masses need not agree.
A path-space limit retains this distinction together with the shapes
of the extreme configurations, whereas same-location summaries do not
in general record it.

We call the resulting notion \emph{multiradial regular variation}
(MrRV). Its defining requirement is a common nonzero limit measure
$\nu$ whenever every component threshold tends to infinity, with no
restriction on their relative rates, namely 
\begin{equation}
 \left(\prod_{i=1}^q b_i(u_i)\right)
 \Pp\{u^{-1}\od X\in\cdot\}
 \xrightarrow{v_\Pi}\nu,
 \qquad u\to\infty,
 \quad b_i\in\RV_{\alpha_i},\quad\alpha_i>0,\quad 1\le i\le q.
 \label{3}
\end{equation}
Here $\RV_\rho$ denotes regular variation at infinity with index
$\rho$, the $b_i$'s are positive, $u^{-1}\od X$ denotes componentwise division, and
$u\to\infty$ means $\min_{1\leq i\leq q}u_i\to\infty$. The notation
$v_\Pi$ denotes convergence under the product localisation specified
in \eqref{182} and \eqref{29}. Thus \eqref{3} is required along every
diverging threshold sequence, not only along a fixed scaling curve.
This is a substantive restriction: different mechanisms may contribute
when thresholds are comparable and when one grows much faster than
the others. Example~\ref{40} exhibits a common shock whose joint
contribution depends on relative threshold growth even though the
conditional scalar path limits do not depend on the realised shock.
Thus conditional scalar path limits alone do not imply MrRV.

The product normalisation forces the limit measure $\nu$ to be homogeneous in each
component separately, namely 
\begin{equation}
 \nu(r\od A)=\left(\prod_{i=1}^q r_i^{-\alpha_i}\right)\nu(A),
 \qquad r\in(0,\infty)^q
 \label{4}
\end{equation}
for every Borel set $A$.\\ 
 A nonzero multihomogeneous measure is a
\emph{multihomogeneous tail measure} if it satisfies the carrier
condition \eqref{12a} and the product-window finiteness condition
\eqref{12b} (Definition~\ref{11}); we abbreviate this to
\emph{tail measure}. The functions $b_i$ normalise the joint tail
layer and need not be marginal tail normalisers. The case $q=1$
recovers scalar regular variation. Independent regularly varying
component fields give tensor-product tail measures. More generally,
separate homogeneity does not imply independence: the stochastic
representer can retain dependence between component shapes and their
relative locations. The multi-component framework of
\cite{Segers2020} instead uses a common scalar action.\\

General scalar actions, including unequal coordinate scalings, are treated in \cite[Section~3.1 and Example~3.1]{LindskogResnickRoy2014}. In that framework the action is indexed by a single positive parameter. Here we consider the full product action of \((0,\infty)^q\) and require a common limiting measure as all component thresholds diverge, without restrictions on their relative rates.\\

Every tail measure admits a stochastic polar representation and a
Pareto realisation with local total-variation convergence
(Proposition~\ref{38}). The main convergence criterion then identifies
which random fields have a prescribed target (Theorem~\ref{44}).
An anchor is a tuple of locations, one for each component. The
criterion combines the normalised probability of a joint exceedance
at that tuple with the conditional law of the entire rescaled field.
At zero-mass anchors only the vanishing of the normalised exceedance
probability is required.\\ 
 In the continuous-parameter setting, fixed
anchors can miss narrow peaks; Example~\ref{227} shows that the
additional compact-window mass bound cannot be omitted even when the
conditional whole-path limits hold.\\

When $\mathsf T$ is an additive group, write
$(B^hf)(t)=f(t-h)$ and call $\nu$ shift-invariant when
\begin{equation}
 \nu\circ(B^h)^{-1}=\nu,
 \qquad h\in\mathsf T.
 \label{13}
\end{equation}
Stationarity of the rf $X$ reduces absolute anchor tuples to relative lags
(Corollary~\ref{50}) but does not in general reduce them to a single
common location.\\ 

 A separate question is whether a shift-invariant
target can itself be realised by a stationary rf. Theorem~\ref{53}
gives a strictly stationary realisation of every shift-invariant tail
measure when the index set is a countable discrete abelian group or
$\mathbb R^m$, $m\ge1$.

Finite moving averages with shared volatility have explicit
joint path-tail measures and relative-anchor masses
(Example~\ref{88}). For every filter length at least four, two such
models have the same one-time volatility distribution, scalar row-tail
measures and simultaneous-exceedance tail masses on every consecutive
finite window, but different tail masses at a nonzero relative lag
(Example~\ref{97}). Thus these scalar and common-time window summaries
do not determine the joint path-tail measure.\\

A componentwise homogeneous mapping lemma gives fixed-window tail
probabilities (Lemma~\ref{104}). Example~\ref{117} also records a
limitation on random multiplication: an independent unbounded
multiplier with finite moments of all positive orders can destroy
MrRV at the original rates under sufficiently unbalanced thresholds.\\

The present paper develops the fixed-window theory. A forthcoming
contribution   investigates growing-window limits through cluster
tail measures on quotient spaces under shifts, including anticlustering
and dependence conditions for Poisson point process convergence and
extremal indices; the one-scale results of
\cite[Sections~2--3]{BojanPhilippe} provide a benchmark.
These questions require additional assumptions and are outside the
scope of this contribution.

Organisation of the rest of this contribution: 
Section~\ref{6} introduces the path space, boundedness and local
random fields. Section~\ref{224} develops multiradial convergence
and its verification criteria, followed by homogeneous mappings.
Section~\ref{225} gives the stationary realisation theorem, and
Section~\ref{61} treats finite observations and finite moving
averages. Proofs are collected in Section~\ref{101}; localisation,
the stationary construction and example calculations are detailed
in Appendix~\ref{129}.

\section{Notation and preliminaries}
\label{6}
We shall introduce our basic notation, the path space and boundedness followed by the definitions of tail measures and local rfs.

\subsection{Path space and boundedness}

Let $q\in\mathbb N=\{1,2,\ldots\}$ and let $\mathsf T$ be a
nonempty finite or countably infinite discrete set, or
$\mathsf T=\mathbb R^m$ with $m\in\mathbb N$. We use the \emph{path space}

\[
   \C=C(\mathsf T,\mathbb R^q)
\]
of continuous functions $f: \mathsf T\to \mathbb R^q$ 
equipped with the topology of locally uniform convergence and its Borel
$\sigma$-field.
Write below $\1_A$ for the indicator of $A$, $\mu(F)=\int F\,d\mu$,
$\1_A\mu$ for the restriction of $\mu$ to $A$, and set 
$\E\{Y;A\}=\E[Y\1_A]$.
Denote further by  $\Law(U)$  the distribution of a random element $U$
and write  $T_\#\mu=\mu\circ T^{-1}$ for the \emph{pushforward} of a
measure $\mu$ under a measurable map $T$.\\ 
For finite Borel measures on a metric space \(E\) the notation
\[\mu_n\Longrightarrow\mu\quad\text{as }n\to\infty\] 
 means
\(\lim_{n\to\infty}\int_E F\,d\mu_n=\int_E F\,d\mu\)
for every bounded continuous \(F:E\to\mathbb R\).
This includes convergence of total masses; for probability laws
it is convergence in distribution. The same convention applies to
threshold-indexed measures as \(u\to\infty\).\\ 
Following the compact-window framework of
\cite[Definition~3.9 and Theorem~4.11]{BladtHashorvaShevchenko2022},
we fix an increasing exhaustion
\((K_n)_{n\geq1}\) of \(\mathsf T\) by nonempty compact sets such that
every compact subset lies in some \(K_n\). When $\mathsf T=\mathbb R^m$, we take
\(K_n=[-n,n]^m\) and throughout we use the \emph{local gauges}
\[
   \rho_{n,i}(f)=\sup_{t\in K_n}|f_i(t)|,\qquad f=(f_1, \ldots, f_q)\in \C, \qquad n\in\mathbb N,\quad 1\le i\le q.
\]
For $\alpha=(\alpha_1,\ldots,\alpha_q)\in(0,\infty)^q$, put
$\G=(0,\infty)^q$ and
\[
 \vartheta_\alpha(dr)=
 \prod_{i=1}^q\alpha_i r_i^{-\alpha_i-1}\,dr_i, \qquad
 \chi_\alpha(r)=\prod_{i=1}^q r_i^{\alpha_i},\qquad
 r=(r_1,\ldots,r_q)\in\G.
\]
We extend $\chi_\alpha$ to $[0,\infty)^q$ by the same product,
and define the \emph{componentwise action} by
\[
   (r\od f)(t)=(r_1f_1(t),\ldots,r_qf_q(t)).
\]

The \emph{deleted set}, \emph{carrier} and \emph{product windows} are
\begin{align}
 D_\Pi&=\bigcup_{i=1}^q\{f\in\C:f_i\equiv0\},
 &\C_\Pi^\circ&=\C\setminus D_\Pi,
 \notag\\
 U_{\boldsymbol n,\varepsilon}
 &=\bigcap_{i=1}^q
   \{f\in\C:\rho_{n_i,i}(f)>\varepsilon_i\},
 &\boldsymbol n&\in\mathbb N^q,
 \quad\varepsilon\in(0,\infty)^q.
 \label{9}
\end{align}
Define the \emph{localisation moduli} and associated open sets by
\begin{equation}
 \tau_n(f)=\min_{1\le i\le q}\rho_{n,i}(f),\qquad
 O_n=\{f\in\C:\tau_n(f)>1/n\},\qquad n\in\mathbb N
 \label{27}
\end{equation}
and hence with this notation 
\begin{equation}
 \C_\Pi^\circ=\bigcup_{n=1}^\infty O_n.
 \label{10}
\end{equation}

We use the distance-from-a-deleted-set formulation of boundedness;
see \cite[Section~2]{LindskogResnickRoy2014} and
\cite[Examples~B.1.6 and B.1.8--B.1.9]{KulikSoulier2020}.
Choose a metric \(d\) on \(\C\) which induces the locally uniform
topology and is invariant under addition of the same path:
\[
 d(f+h,g+h)=d(f,g),\qquad f,g,h\in\C.
\]
This invariance concerns addition of paths independently of
translations of the index set. One such metric is
\begin{equation}
 d_0(f,g)=\sum_{n=1}^\infty 2^{-n}
       \left(1\wedge\max_{1\le i\le q}\rho_{n,i}(f-g)\right).
 \label{190}
\end{equation}
The set \(D_\Pi\) is closed. For a Borel set $A\subset\C_\Pi^\circ$, put
\[
 \delta_d(f)=\operatorname{dist}_d(f,D_\Pi),\qquad f\in \C,\qquad
 \operatorname{dist}_d(A,D_\Pi)=\inf_{f\in A}\delta_d(f), 
\]
with \(\inf\varnothing=\infty\), and define
\begin{equation}
 \mathcal B_\Pi=\left\{A\subset\C_\Pi^\circ:
       A\text{ is Borel and }
       \operatorname{dist}_d(A,D_\Pi)>0\right\}.
 \label{182}
\end{equation}
This collection is independent of the chosen compatible metric
invariant under addition. For Borel \(A\subset\C_\Pi^\circ\) we have 
\begin{align}
 A\in\mathcal B_\Pi
 &\quad\Longleftrightarrow\quad
 A\subset U_{\boldsymbol n,\varepsilon}
 \quad\text{for some }\boldsymbol n\in\mathbb N^q,
       \ \varepsilon\in\G
 \label{192}\\
 &\quad\Longleftrightarrow\quad
 \inf_{f\in A}\tau_n(f)>0\quad\text{for some }n\in\mathbb N.
 \label{28}
   \end{align}
Thus boundedness requires each component to exceed a fixed positive
level somewhere on a fixed compact window. The witnessing location
may depend on the component and the path; no upper bound on values
is imposed. The sets \(O_n\) increase, eventually contain every
\(A\in\mathcal B_\Pi\), and satisfy
\(\overline O_n\subset O_{n+1}\).

We call members of \(\mathcal B_\Pi\) \emph{bounded} and a Borel
measure on \(\C_\Pi^\circ\) \emph{boundedly finite} if it is finite
on every such set, equivalently on every product window \eqref{9}.
Componentwise scaling preserves this boundedness, namely 
\begin{equation}
 A\in\mathcal B_\Pi\quad\Longleftrightarrow\quad
 r\od A\in\mathcal B_\Pi,\qquad r\in\G.
 \label{131}
\end{equation}
The metric independence and localisation properties are proved in
Appendix~\ref{129}.

\subsection{Tail measures and local rfs}
Below we set  $\boldsymbol1=(1,\ldots,1)\in \mathbb R^q$.
\begin{definition}
\label{11}
A nonzero Borel measure \(\nu\) on \(\C\) is an
\emph{$\alpha$-multihomogeneous tail measure} if
\begin{subequations}
\label{12}
\begin{align}
 \nu(D_\Pi)&=0,
 \label{12a}\\
  \nu(U_{\boldsymbol n,\boldsymbol1})&<\infty,
 \qquad \boldsymbol n\in\mathbb N^q,
 \label{12b}\\
 \nu(r\od A)&=\chi_\alpha(r)^{-1}\nu(A),
 \qquad r\in\G,\quad A\subset\C\text{ Borel}.
 \label{12c}
\end{align}
\end{subequations}

It is \emph{shift-invariant} if \eqref{13} also holds.
\end{definition}
 By \eqref{12c}
\[
 \nu(U_{\boldsymbol n,\varepsilon})
 =\chi_\alpha(\varepsilon)^{-1}
  \nu(U_{\boldsymbol n,\boldsymbol1}),
 \qquad \varepsilon\in\G,
\]
so \eqref{12b} ensures finiteness on every product window. Hence in view of \eqref{10} 
every tail measure is \(\sigma\)-finite.  \\ 
The term \emph{multihomogeneous measure} refers to the scaling identity
\eqref{12c} alone; the tail-measure class also requires the carrier
condition \eqref{12a} and product-window finiteness \eqref{12b}. This is the path-space
specialisation of the structural terminology in
\cite{Hashorva2026Multihomogeneous,Hashorva2026Particles}.
The word \emph{product} in the action and localisation does not assert
factorisation of $\nu$; we reserve \emph{tensor-product tail measure}
for a measure of the form $\bigotimes_{i=1}^q\nu_i$.
Although \eqref{12c} implies
$(-\sum_{i=1}^q\alpha_i)$-homogeneity under diagonal scaling,
\eqref{12b} does not imply finiteness on ordinary origin-separated sets. For example, the positive tensor-product Pareto measure in
two dimensions assigns mass one to $\{x_1>1,x_2>1\}$ and infinite mass
to $\{x_1>1\}$.

 Throughout the paper we fix a
\(\C_\Pi^\circ\)-valued random element \(Z\) satisfying the compact
mixed-moment condition \eqref{18} and work with
\begin{equation}
 \nu:=\nu_Z,\qquad
 \nu_Z(A)=\E\left\{\int_\G\1_{\{r\od Z\in A\}}\,
              \vartheta_\alpha(dr)\right\}.
 \label{14}
\end{equation}
An \emph{anchor tuple} is a tuple
\(\boldsymbol h=(h_1,\ldots,h_q)\in\mathsf T^q\), with one
observation location for each component. Its \emph{anchor event} and
\emph{anchor mass} are
\begin{equation}
 A_{\boldsymbol h}
 =\{f\in\C:|f_i(h_i)|>1,\ 1\leq i\leq q\},
 \qquad
 p_{\boldsymbol h}=\nu(A_{\boldsymbol h}).
 \label{15}
\end{equation}
The anchor mass satisfies \(0\le p_{\boldsymbol h}<\infty\) by
\eqref{12b} and it may be zero. When \(p_{\boldsymbol h}>0\)  the
\emph{local tail rf} \(Y^{[\boldsymbol h]}\) and its \emph{spectral rf}
\(\Theta^{[\boldsymbol h]}\) are defined by
\begin{equation}
 \Pp\{Y^{[\boldsymbol h]}\in A\}
 =p_{\boldsymbol h}^{-1}\nu(A\cap A_{\boldsymbol h}),
 \qquad
 \Theta_i^{[\boldsymbol h]}(t)
 =\frac{Y_i^{[\boldsymbol h]}(t)}
 {|Y_i^{[\boldsymbol h]}(h_i)|}.
 \label{16}
\end{equation}

The elementary construction from countably many evaluation sectors in
Proposition~\ref{38} represents every measure in
Definition~\ref{11} in the form
\eqref{14} so that the standing convention
\(\nu=\nu_Z\) entails no loss of generality within this class; 
the random path $Z$ is called a \emph{stochastic representer} of $\nu$. Direct radial integration yields for all $\varepsilon\in (0,\infty)^q$
\begin{equation}
 \nu(U_{\boldsymbol n,\varepsilon})
 =\E\left\{\prod_{i=1}^q
 \left(\frac{\rho_{n_i,i}(Z)}{\varepsilon_i}\right)^{\alpha_i}\right\}
 \label{17}
\end{equation}
and therefore  product-window finiteness \eqref{12b} is equivalent to
\begin{equation}
 \E\left\{\prod_{i=1}^q\rho_{n_i,i}(Z)^{\alpha_i}\right\}<\infty,
 \qquad \boldsymbol n\in\mathbb N^q.
 \label{18}
\end{equation}

The \emph{visible sector} at \(\boldsymbol h\) 
denoted by $E_{\boldsymbol h}=\{f \in\C:|f_i(h_i)|>0,\ 1\leq i\leq q\}$ consists of paths
whose components are nonzero at their respective anchor locations. Writing next 
\begin{equation} 
 w_{\boldsymbol h}(f)=\prod_{i=1}^q|f_i(h_i)|^{\alpha_i},\qquad
 (N_{\boldsymbol h}f)_i=\frac{f_i}{|f_i(h_i)|},
 \quad f\in E_{\boldsymbol h},\quad 1\le i\le q 
\label{183}
\end{equation}
and substituting \(s_i=r_i|Z_i(h_i)|\), \(1\le i\le q\), in \eqref{14} for every
nonnegative Borel \(F\) we obtain
\[
 \nu(F\1_{E_{\boldsymbol h}})
 =\E\left\{w_{\boldsymbol h}(Z)
   \int_\G F(s\od N_{\boldsymbol h}Z)\,
      \vartheta_\alpha(ds);\ Z\in E_{\boldsymbol h}\right\}.
\]
Restricting the radial integral to \((1,\infty)^q\) yields
\begin{align}
 p_{\boldsymbol h}&=\E\{w_{\boldsymbol h}(Z)\}<\infty,
 \label{19}\\
 p_{\boldsymbol h}\E\{F(\Theta^{[\boldsymbol h]})\}
 &=\E\bigl\{w_{\boldsymbol h}(Z)
 F(N_{\boldsymbol h}Z);\ Z\in E_{\boldsymbol h}\bigr\},
 \qquad F\geq0,
 \label{20}
\end{align}
where all functionals are Borel and the second identity requires
\(p_{\boldsymbol h}>0\). We define the integrand in \eqref{20} to be
zero outside \(E_{\boldsymbol h}\) before performing any division. At
an anchor of positive mass we also have
\begin{equation*}
 Y^{[\boldsymbol h]}\stackrel d=R\od\Theta^{[\boldsymbol h]},
\end{equation*}
where \(R=(R_1,\ldots,R_q)\) is independent of
\(\Theta^{[\boldsymbol h]}\) and has mutually independent coordinates
satisfying
\begin{equation}
 \Pp\{R_i>x\}=x^{-\alpha_i},\qquad x\geq1,\quad1\leq i\leq q.
 \label{22}
\end{equation}
The unrestricted integral and the spectral tilt similarly yield
\begin{equation}
 \nu(A\cap E_{\boldsymbol h})
 =p_{\boldsymbol h}\E\left\{\int_\G
 \1_{\{r\od\Theta^{[\boldsymbol h]}\in A\}}
 \vartheta_\alpha(dr)\right\}.
 \label{23}
\end{equation}
Since
\(E_{\boldsymbol h}=\bigcup_{n=1}^{\infty}
 ((1/n,\ldots,1/n)\od A_{\boldsymbol h})\), homogeneity implies
\(\nu(E_{\boldsymbol h})=0\) when \(p_{\boldsymbol h}=0\), without
introducing a local probability law. When \(p_{\boldsymbol h}>0\),
\eqref{23} implies \(\nu(E_{\boldsymbol h})=\infty\); the finite mass
used in \eqref{16} is \(p_{\boldsymbol h}=\nu(A_{\boldsymbol h})\).

Fix now the countable dense set
\begin{equation}
 \mathsf D=
 \begin{cases}
 \mathsf T,&\mathsf T\text{ is discrete},\\
 \mathbb Q^m,&\mathsf T=\mathbb R^m,
 \end{cases}
 \qquad
 \C_\Pi^\circ=\bigcup_{\boldsymbol h\in\mathsf D^q}E_{\boldsymbol h}.
 \label{184}
\end{equation}
Each component of \(f\in\C_\Pi^\circ\) is continuous and not
identically zero, hence it is nonzero at some point of the dense
set \(\mathsf D\). Choosing these points componentwise gives
\(f\in E_{\boldsymbol h}\) for some
\(\boldsymbol h\in\mathsf D^q\), proving \eqref{184}.
Thus the anchor restrictions
\(\{\1_{A_{\boldsymbol h}}\nu:\boldsymbol h\in\mathsf D^q\}\),
equivalently their masses and positive-mass spectral laws, determine
\(\nu\) by \eqref{23} and the countable cover \eqref{184}. Direct integration in \eqref{14} also yields for all $n\in\mathbb N$
\begin{equation}
 \nu\{f\in\C:|f_i(t)|=c\}
 =\nu\{f\in\C:\rho_{n,i}(f)=c\}=0,
 \qquad c>0,\quad t\in\mathsf T,\quad1\leq i\leq q.
 \label{25}
\end{equation}

\section{Main Results}
\label{224}
We first define MrRV and give a Pareto realisation of every tail
measure, followed by convergence criteria and a homogeneous mapping
principle. Finally, we construct a strictly stationary realisation of
every shift-invariant tail measure when the index set is a countable
discrete abelian group or $\mathbb R^m$.

\subsection{Definition and Pareto realisation}

We use \(\C_\Pi^\circ\) as the convergence carrier   and identify each measure on this carrier with its zero extension
to \(\C\).

Let \(C_\Pi(\C)\) denote the bounded continuous functions
\(F:\C\to\mathbb R\) for which there are
\(\boldsymbol n\in\mathbb N^q\) and \(\varepsilon\in\G\) such that
\begin{equation}
   F(f)=0
   \quad\text{if }\rho_{n_i,i}(f)\leq\varepsilon_i
   \text{ for at least one }i\in\{1,\ldots,q\}, \qquad f\in \C.
   \label{29}
\end{equation}

When \(\mathsf T\) is an additive group, each shift \(B^h\) is a
homeomorphism of \(\C\). Since every compact translate \(K_n\pm h\)
lies in some \(K_m\), \eqref{192} and \eqref{29} imply
\[
 A\in\mathcal B_\Pi\quad\Longleftrightarrow\quad B^hA\in\mathcal B_\Pi,
 \qquad
 F\in C_\Pi(\C)\quad\Longleftrightarrow\quad F\circ B^h\in C_\Pi(\C).
\]

For \(u=(u_1,\ldots,u_q)\in\G\), write
\begin{equation}
   u\to\infty
   \quad\Longleftrightarrow\quad
   \min_{1\leq i\leq q} u_i\to\infty.
   \label{30}
\end{equation}
In the sequel all threshold limits are understood along every sequence
\(u^{(k)}\) with \(\min_{1\leq i\leq q}u_i^{(k)}\to\infty\)
as \(k\to\infty\), equivalently as
limits with all component thresholds sufficiently large.
In particular, we define
\[
 \limsup_{u\to\infty}g(u)
 :=\inf_{M>0}\ \sup_{\substack{u\in\G\\ \min_{1\leq i\leq q}u_i\ge M}}g(u).
\]
For a family \((\zeta_u)_{u\in\G}\) of boundedly finite measures and a
boundedly finite measure \(\zeta\), write
\(\zeta_u\xrightarrow{v_\Pi}\zeta\), as \(u\to\infty\), if
\[
   \lim_{u\to\infty}\int F\,d\zeta_u=\int F\,d\zeta,
   \qquad F\in C_\Pi(\C).
\]

Let $a(u)=\prod_{i=1}^q b_i(u_i)$ where
\[
   b(u)=(b_1(u_1),\ldots,b_q(u_q)),
   \qquad b_i:\Rplus\to\Rplus,\qquad
   b_i\in\RV_{\alpha_i},\qquad 1\le i\le q.
\] 
For a \(\C\)-valued rf   \(X\) define its normalised laws
on \(\C_\Pi^\circ\) by
\begin{equation}
   \mu_u^{X,b}(A)
   =a(u) 
   \Pp\{u^{-1}\od X\in A\},
   \qquad A\ \text{Borel in }\C_\Pi^\circ.
   \label{31}
\end{equation}
\begin{definition}
\label{32}
The rf \(X\) is MrRV with tail measure \(\nu\) of index
\(\alpha\) and scaling \(b\),   if \(\nu\) is
nonzero and boundedly finite and
\begin{equation}
   \mu_u^{X,b}\xrightarrow{v_\Pi}\nu,
   \qquad u\to\infty.
   \label{33}
\end{equation}
\end{definition}

We shall abbreviate \eqref{33} as
\[
   X\in\RV_\Pi(\alpha,\nu;b).
\]

Applying \eqref{33} after a fixed componentwise rescaling and using
regular variation of the $b_i$ together with \eqref{131} gives the
homogeneity identity \eqref{12c}. Thus every limit in
Definition~\ref{32}, extended by zero to $D_\Pi$, is a tail measure in
Definition~\ref{11}. If $\mathsf T$
is an additive group and $X$ is strictly stationary, shifting the
normalised laws similarly establishes the shift-invariance property  \eqref{13}.

The \(b_i\) normalise the joint layer and need not be marginal tail
normalisers: events where other components are invisible lie outside
product localisation.

The componentwise multiplicative normalisation is not identifiable: if
\(\widetilde b_i=c_i b_i\) with \(c_i>0\) for \(1\le i\le q\), then the same rf $X$  has tail
measure \((\prod_{i=1}^q c_i)\nu\) under the scaling 
\(\widetilde b\). We do not require any of the anchor masses \eqref{15} to equal one. For a
stationary nonzero target with $q=1$, the origin mass is positive and
can be normalised to one. For $q\ge2$, a nonzero target may have
$p_{(0,\ldots,0)}=0$ because different components need not be visible
at the same location. For example, on \(\mathbb Z\) with \(q=2\),
the shift-invariant tail measure
\[
 \nu(F)=\sum_{s\in\mathbb Z}\int_\G
 F\bigl(r_1\1_{\{s\}},r_2\1_{\{s+1\}}\bigr)\,
 \vartheta_\alpha(dr),\qquad F\ge0,
\]
has \(p_{(0,0)}=0\) and \(p_{(0,1)}=1\). Only finitely many
centres \(s\) contribute to each product window defined in \eqref{9}, so \eqref{12b} holds.

For the canonical power tail rates, put
\begin{equation}
   b_\alpha(u)
   =(u_1^{\alpha_1},\ldots,u_q^{\alpha_q}).
   \label{37}
\end{equation}
When \(b=b_\alpha\), we shall write
\[
   \mu_u^X(A):=\mu_u^{X,b_\alpha}(A)
   =\chi_\alpha(u)\Pp\{u^{-1}\od X\in A\},
\]
and write \[X\in\RV_\Pi(\alpha,\nu)\]
 in place of
\(X\in\RV_\Pi(\alpha,\nu;b_\alpha)\).
This direct-level convention uses \(u_i\) as the actual threshold.
If \(e_i\in\RV_{1/\alpha_i}\) is an asymptotic inverse of \(b_i\) with
\(\lim_{s\to\infty}b_i(e_i(s))/s=1\) for \(1\le i\le q\), then the substitution \(u_i=e_i(s_i)\) recovers
the asymptotically equivalent quantile formulation with multiplier
\(\prod_{i=1}^q s_i\) and thresholds \(e_i(s_i)\), \(1\le i\le q\), as \(\min_{1\le i\le q}s_i\to\infty\).

For the realisation below let \(R\) be the Pareto vector of
\eqref{22}, independent of the representer \(Z\). For finite measures
\(\mu,\eta\) on a common measurable space \((E,\mathcal E)\), write
\[
 \|\mu-\eta\|_{\mathrm{TV}}
 :=\sup_{\substack{F:E\to[-1,1]\\ F\ \mathrm{measurable}}}
       \left|\int F\,d\mu-\int F\,d\eta\right|.
\]

\begin{proposition}
\label{38}
Every measure in Definition~\ref{11}
equals $\nu_Z$ for a $\C_\Pi^\circ$-valued random element $Z$
satisfying \eqref{18}. If $\nu$ is carried by a
Borel set invariant under componentwise positive scaling, $Z$ can be
chosen in that set almost surely. For every such representer we have
\begin{align}
 X&=R\od Z\in\RV_\Pi(\alpha,\nu_Z),
 \label{39}\\
 \lim_{u\to\infty}
 \left\|\1_A\mu_u^X-\1_A\nu\right\|_{\mathrm{TV}}&=0,
 \qquad A\in\mathcal B_\Pi.
 \notag
\end{align}
\end{proposition}

Separate radial coordinates and mixed compact moments extend the
one-scale representation and Pareto construction in
\cite[Lemma~3.10 and Example~4.17]{BladtHashorvaShevchenko2022}; see also
\cite[Theorem~2.4]{DombryHashorvaSoulier2018}. Exact-Pareto truncation
gives local total-variation convergence (proof in Section~\ref{101}).
Theorem~\ref{53} supplies a stationary realisation for every
shift-invariant target when the index set is a countable discrete
abelian group or $\mathbb R^m$.

\begin{example}
\label{40}
Take \(q=2\) and \(\alpha_1=\alpha_2=\alpha>0\) and let
\(R_1,R_2,S\) be independent standard Pareto variables with respective
indices \(\alpha,\alpha,2\alpha\). Set \(X_i=\max(R_i,S)\), \(i\in\{1,2\}\), and regard
these variables as constant paths. Conditional on \(S\) the components
are independent and their scalar limits as \(u\to\infty\), under \(b_i(u)=u^\alpha\), \(i\in\{1,2\}\), are
the usual Pareto measures which do not depend on \(S\). The product of
these conditional limits has unit mass on \(\{f_1>1,f_2>1\}\) whereas
the threshold paths \(u=(t,t)\) and \(u=(t,t^2)\) give
\begin{align*}
 \lim_{t\to\infty}t^{2\alpha}\Pp\{X_1>t,X_2>t\}
 &=\lim_{t\to\infty}(2-t^{-2\alpha})=2,\\
 \lim_{t\to\infty}t^{3\alpha}\Pp\{X_1>t,X_2>t^2\}
 &=\lim_{t\to\infty}
 (1-t^{-2\alpha}+2t^{-\alpha}-t^{-3\alpha})=1.
\end{align*}
The normalised mass of \(\{f_1>1,f_2>1\}\) has no limit as \(u\to\infty\)
because relative threshold growth detects the additional common-shock
contribution. Thus conditional scalar path limits alone do not imply
MrRV.
The ordinary one-scale tail measure of $X$ is nevertheless the same as
that of \((R_1,R_2)\), since
\(\lim_{t\to\infty}t^\alpha\Pp\{S>\varepsilon t\}=0\)
for every \(\varepsilon>0\). Thus ordinary regular variation does
not determine the multiradial layer.
\end{example}

\subsection{Convergence criteria}

For the criterion below let \(X\) be a \(\C\)-valued rf and
retain the prescribed target measure \(\nu=\nu_Z\). Write
\(\mu_u=\mu_u^{X,b}\) and recall
\(O_n=\{\tau_n>1/n\}\), with \(\tau_n\) from \eqref{27}.
All limits in \(u\) are taken as \(u\to\infty\) in the sense of
\eqref{30}, without restricting the relative threshold rates. Use
\(\mathsf D\) from \eqref{184} and for \(\boldsymbol h\in\mathsf T^q\), \(x\in\G\)  define
\begin{equation}
 A_{\boldsymbol h,x}
 =\{f \in\C:|f_i(h_i)|>x_i,\ 1\leq i\leq q\}.
 \label{42}
\end{equation}
In this notation, we have $A_{\boldsymbol h,\boldsymbol1}=A_{\boldsymbol h}$, with $A_{\boldsymbol h}$ defined previously.
In view of \eqref{25} continuity of the anchor and compact-supremum maps implies
\begin{equation}
 \nu(\partial A_{\boldsymbol h,x})=\nu(\partial O_n)=0.
 \label{43}
\end{equation}

\begin{theorem}
\label{44}  $X\in\RV_\Pi(\alpha,\nu;b)$ if and only if as $u\to\infty$
\begin{equation}
 \1_{A_{\boldsymbol h}}\mu_u
 \Longrightarrow\1_{A_{\boldsymbol h}}\nu
 \quad\text{as finite measures on }\C,
 \qquad \boldsymbol h\in\mathsf D^q
 \label{45}
\end{equation}
and when \(\mathsf T=\mathbb R^m\)
\begin{equation}
 \limsup_{u\to\infty}\mu_u(O_n)\leq\nu(O_n),
 \qquad n\geq1.
 \label{46}
\end{equation}
In the discrete case \eqref{46} follows from \eqref{45}.
\end{theorem}

Under \eqref{45}, the bound \eqref{46} is equivalent to
\(\lim_{u\to\infty}\mu_u(O_n)=\nu(O_n)\) for every \(n\ge1\). For an anchor with
\(p_{\boldsymbol h}=0\), \eqref{45} means just
\(\lim_{u\to\infty}\mu_u(A_{\boldsymbol h})=0\).  For
\(p_{\boldsymbol h}>0\), it is equivalent to the following
two limits as \(u\to\infty\):
\begin{align}
 \lim_{u\to\infty}a(u)\Pp\{|X_i(h_i)|>u_i,\ 1\leq i\leq q\}
 &=p_{\boldsymbol h},
 \label{47}\\
 \Law\left(u^{-1}\od X\ \middle|\
 |X_i(h_i)|>u_i,\ 1\leq i\leq q\right)
 &\Longrightarrow\Law(Y^{[\boldsymbol h]}),
 \label{48}
\end{align}
see  
\cite[Theorems~4.7, 4.15 and Remark~4.16(ii)]{BladtHashorvaShevchenko2022} for the case $q=1$.
Whole-path convergence in \eqref{48} includes conditional tightness;
\eqref{46} controls compact-window peaks missed by finitely many
anchor tuples. The proof is in Section~\ref{101}.
 
If \eqref{46} is imposed in either index setting, it suffices to
verify \eqref{45} at anchors with \(p_{\boldsymbol h}>0\).
Indeed, fix \(n\in\mathbb N\) and enumerate the cover of \(O_n\) by
\(A_{\boldsymbol h,(1/n,\ldots,1/n)}\),
\(\boldsymbol h\in(\mathsf D\cap K_n)^q\), as \((A_j)_{j\ge1}\), repeating sets if the cover is finite. Put
\(V_m^+=\bigcup_{1\le j\le m:\,\nu(A_j)>0}A_j\).
Scaling as in \eqref{118} and disjointifying give
\(\1_{V_m^+}\mu_u\Longrightarrow\1_{V_m^+}\nu\) as
\(u\to\infty\), for fixed \(m\), and
\(\nu(V_m^+)\uparrow\nu(O_n)\) as \(m\to\infty\).
Hence \eqref{46} yields
\[
 \limsup_{u\to\infty}\mu_u(O_n\setminus V_m^+)
 \le \nu(O_n)-\nu(V_m^+)\longrightarrow0,
 \qquad m\to\infty.
\]
Approximation yields convergence on \(O_n\) as \(u\to\infty\), and Lemma~\ref{133}
then implies \eqref{45} also at zero-mass anchors.

In continuous time, choose for every \(n\) increasing nonempty finite
sets \(G_{n,k}\subset\mathsf D\cap K_n\), \(k\ge1\), with dense union and set
\[
 C_{n,k}
 =\left\{f\in\C:\min_{1\le i\le q}\max_{t\in G_{n,k}}|f_i(t)|
                  >\frac1{2n}\right\}.
\]
Under \eqref{45} at all anchors, condition \eqref{46} is equivalent to
\[
 \lim_{k\to\infty}\limsup_{u\to\infty}
 \mu_u(O_n\setminus C_{n,k})=0,\qquad n\ge1.
\]
Indeed, \(C_{n,k}\) is a finite union of anchor events at level
\(1/(2n)\). Scaling \eqref{45} as in \eqref{118}, disjointifying
this union and using \eqref{43} imply
\(\lim_{u\to\infty}\mu_u(O_n\cap C_{n,k})=\nu(O_n\cap C_{n,k})\).
Decomposing \(O_n\) into this intersection and its remainder proves
\eqref{46}. Conversely, \eqref{45} and \eqref{46} give convergence
of \(\mu_u(O_n)\) as \(u\to\infty\), so subtraction yields
\(\lim_{u\to\infty}\mu_u(O_n\setminus C_{n,k})=\nu(O_n\setminus C_{n,k})\)
for fixed \(n,k\).
The latter decreases to zero as \(k\to\infty\), since
\(O_n\subset\bigcup_{k=1}^\infty C_{n,k}\).

We show next that the compact-window condition \eqref{46} cannot be omitted in
continuous time.
\begin{example} 
\label{227}
 Let \(\mathsf T=\mathbb R\), let \(R\) be the Pareto
vector in \eqref{22}, let \(V\) be a standard Pareto variable of
index \(\alpha_1\), and let \(U\) be uniform on \([0,1]\), with
\(R,V,U\) mutually independent. Put
\[
 P(t)=V(1-V|t-U|)_+,\qquad
 X_1(t)=\max\{R_1,P(t)\},\qquad
 X_i(t)=R_i\quad(2\le i\le q).
\]
Use the power tail rates \eqref{37}, and let \(\nu=\nu_Z\) with
\(Z_i(t)=1\) for every \(1\le i\le q\) and \(t\in\mathbb R\), the tensor-product Pareto measure
on positive constant paths. For each fixed \(h\in\mathbb R\) and \(u\ge1\)
\[
 \Pp\{P(h)>u\}
 \le 2\E\{V^{-1};V>u\}
 =\frac{2\alpha_1}{\alpha_1+1}u^{-\alpha_1-1}.
\]
Fix an anchor tuple \(\boldsymbol h\), and write
\(B_u=\{R_i>u_i,\ 1\le i\le q\}\).
The event \(B_u\) is contained in
\(\{u^{-1}\od X\in A_{\boldsymbol h}\}\), and independence yields
\[
 \chi_\alpha(u)
 \Pp\{u^{-1}\od X\in A_{\boldsymbol h},\ B_u^c\}
 \le u_1^{\alpha_1}\Pp\{P(h_1)>u_1\}
 \longrightarrow0, \qquad u\to \infty.
\]
Conditional on \(B_u\), the vector \(u^{-1}\od R\) has the law
of \(R\), while the supremum over \(\mathbb R\) of the difference
between the first components of \(u^{-1}\od X\) and
\(u^{-1}\od R\) is at most \(V/u_1\), which tends to zero
in conditional probability as \(u\to\infty\). Thus \eqref{45} holds at every
anchor, including convergence of the conditional whole-path laws.

Nevertheless, \(K_n=[-n,n]\) contains \(U\) and hence for \(u_i\ge n\), \(1\le i\le q\),
\[
 \mu_u(O_n)
 =2n^{\sum_{i=1}^q\alpha_i}
  -n^{\sum_{i=1}^q\alpha_i+\alpha_1}u_1^{-\alpha_1}
 \longrightarrow
 2n^{\sum_{i=1}^q\alpha_i}
 =2\nu(O_n), \qquad u\to\infty.
\]
Hence \eqref{46} fails. Since the anchor restrictions determine
any tail-measure limit, \(X\) admits no such limit at these
tail rates.
\end{example}

Next, for each \(n\ge1\) set  \[
 B_n=\{f:\rho_{n,i}(f)>1,\ 1\le i\le q\},\qquad
 c_n=\nu(B_n),\qquad
 (S_nf)_i=\frac{f_i}{\rho_{n,i}(f)},
\]
where \(S_n\) is defined when all denominators are positive.

\begin{proposition} 
\label{prop:local-polar}
For each \(n\ge1\) with \(c_n>0\), define
\(Q_n=c_n^{-1}(S_n)_\#(\1_{B_n}\nu)\). 
Then \(X\in\RV_\Pi(\alpha,\nu;b)\) if and only if for every \(n\ge1\)
\begin{equation}
\lim_{u\to\infty} a(u)\Pp\{\rho_{n,i}(X)>u_i,\ 1\le i\le q\}= c_n,
 \label{eq:window-polar-mass}
\end{equation}
and whenever \(c_n>0\)
\begin{equation}
 \Law\left(S_nX\ \middle|\ \rho_{n,i}(X)>u_i,\ 1\le i\le q\right)
 \Longrightarrow Q_n\quad\text{on }\C, \qquad u\to \infty.
 \label{eq:window-polar-angle}
\end{equation}
For such \(n\), these conditions also imply
\begin{equation}
 \Law\left(\left((\rho_{n,i}(X)/u_i)_{i=1}^q,S_nX\right)
 \ \middle|\ \rho_{n,i}(X)>u_i,\ 1\le i\le q\right)
 \Longrightarrow\Law(R)\otimes Q_n,\qquad u\to\infty.
 \label{eq:window-polar-joint}
\end{equation}
where \(R\) is the Pareto vector in \eqref{22}.
\end{proposition}

The angular laws \(Q_n\) describe whole paths normalised on \(K_n\),
not merely their restrictions to that window. All limits are cofinal
in the full threshold vector. This is a local multiradial counterpart
of the scalar polar criterion in
\cite[Proposition~3.1]{segers2017polar}: independent scaling coordinates
replace the single radius. In continuous time no separate condition
\eqref{46} is needed here, since window masses are already controlled by
\eqref{eq:window-polar-mass}; whole-path angular convergence still
includes conditional tightness. The proof is in Section~\ref{101}.

The local-tail-field formulation is \eqref{47}--\eqref{48}.
Under \eqref{47}, the same argument with the gauges
\(|f_i(h_i)|\) shows that \eqref{48} may equivalently be
replaced by
\begin{equation}
 \Law\left(N_{\boldsymbol h}X\ \middle|\
 |X_i(h_i)|>u_i,\ 1\le i\le q\right)
 \Longrightarrow\Law(\Theta^{[\boldsymbol h]}), \qquad u\to \infty.
 \label{eq:anchored-polar-angle}
\end{equation}
Thus Theorem~\ref{44} gives both the local-tail-field and
anchored spectral formulations. For \(q=1\) and \(\mathsf T=\mathbb Z\),
these are the two equivalent criteria in
\cite[Lemma~3.5(ii)--(iii)]{DombryHashorvaSoulier2018};
see also \cite[Proposition~2.7]{DombryHashorvaSoulier2018}
for the Pareto--spectral factorisation.
The spectral-process criterion in
\cite[Theorem~5.1]{segers2017polar} treats $q=1$ in the stationary setting.   

For an additive index group $\mathsf T $ write
\(\ell=(\ell_2,\ldots,\ell_q)\in\mathsf D^{q-1}\) and put
\begin{equation}
 A_{\ell,x}:=A_{(0,\ell),x},
 \qquad A_\ell:=A_{(0,\ell)},
 \label{49}
\end{equation}
using the anchor events from \eqref{15} and \eqref{42}.

\begin{corollary}
\label{50}
Suppose that \(X\) is strictly stationary and that \(\mathsf T\) is an
additive group. Then \(X\in\RV_\Pi(\alpha,\nu;b)\) if and only if
\(\nu\) is shift-invariant, the convergence
\begin{equation}
 \1_{A_\ell}\mu_u\Longrightarrow\1_{A_\ell}\nu
 \quad\text{as finite measures on }\C,
 \qquad \ell\in\mathsf D^{q-1}
 \label{51}
\end{equation}
holds as \(u\to\infty\), and the compact-window mass condition \eqref{46} holds when
\(\mathsf T=\mathbb R^m\).
\end{corollary}

At each positive-mass lag, \eqref{51} may equivalently be checked
through the mass and conditional whole-path limits
\eqref{47}--\eqref{48}, or equivalently through \eqref{47} and
\eqref{eq:anchored-polar-angle}.
At a zero-mass lag it requires only vanishing normalised anchor mass.
For \(q=1\), only the origin anchor remains, as in
\cite[Corollary~2.8]{DombryHashorvaSoulier2018} and
\cite[Theorems~2.3 and~3.2]{Soulier2022}.
For \(q\ge2\), the common-location lag \(\ell=\boldsymbol0\)
is generally insufficient. Shift invariance also gives, for
\(\boldsymbol h\in\mathsf T^q\), \(a\in\mathsf T\), and every
positive-mass anchor
\begin{equation}
 p_{\boldsymbol h+a}=p_{\boldsymbol h},\qquad
 \bigl(Y^{[\boldsymbol h+a]},\Theta^{[\boldsymbol h+a]}\bigr)
 \stackrel d=
 \bigl(B^aY^{[\boldsymbol h]},B^a\Theta^{[\boldsymbol h]}\bigr),
 \label{52}
\end{equation}
where \(\boldsymbol h+a=(h_1+a,\ldots,h_q+a)\).
The proof is in Section~\ref{101}.

\begin{example} 
\label{231}
Fix \(\alpha>0\). Let \((\eta_{i,t})_{i=1,2,\,t\in\mathbb Z}\) be mutually independent
standard \(\alpha\)-Pareto variables and set
\[
 X_1(t)=\eta_{1,t},\qquad X_2(t)=\eta_{1,t-1}+\eta_{2,t}, \quad t\in \mathbb Z.
\]
The process $X$ is strictly stationary. At each common time its coordinates
are independent, and the second has tail constant two, hence 
\[
 \lim_{u\to\infty}u_1^\alpha u_2^\alpha
 \Pp\{X_1(0)>u_1x,\ X_2(0)>u_2y\}
 =2x^{-\alpha}y^{-\alpha},\qquad x,y>0.
\]
At lag one, however, a single innovation drives both exceedances, namely 
\[
 v^{2\alpha}\Pp\{X_1(0)>v,\ X_2(1)>v\}
 \ge v^\alpha\longrightarrow\infty,\qquad v\to\infty.
\]
Hence MrRV fails under \(b_i(v)=v^\alpha\), \(i=1,2\).
One-time vector convergence therefore cannot replace the
conditional whole-path and relative-anchor checks above.
\end{example}

\subsection{Homogeneous mappings}

The componentwise mapping principle below applies to all index sets in
Section~\ref{6} without shift invariance. For the scalar principle see
\cite[Proposition~2.1.12 and Theorem~B.1.21]{KulikSoulier2020} and
\cite[Lemma~6.1 and Theorem~A.2]{BladtHashorvaShevchenko2022}; for applications see
\cite{dyszewski2020homogeneous}. The measure mapping principle for nets
is given in \cite[Lemma~3.1.3 and Remark~3.1.4]{BasrakMilincevicMolchanov2025}.

Let \(X\in\RV_\Pi(\alpha,\nu;b)\) and let \(H:\C\to\mathbb R^q\)
be Borel measurable and continuous \(\nu\)-almost everywhere with
\[
 H(r\od f)=r\od H(f),\qquad r\in\G,\ f\in\C.
\]
For \(1\le i\le q\), equivariance forces \(H_i(f)=0\) whenever \(f_i\equiv0\), by
varying only the \(i\)-th scaling coordinate. Thus the set \(A_H\)
below is automatically contained in \(\C_\Pi^\circ\); its uniform
separation from \(D_\Pi\) is the additional locality requirement.
We assume the locality and positivity conditions
\begin{equation}
 A_H=\{f\in\C:|H_i(f)|>1,\ 1\leq i\leq q\}\in\mathcal B_\Pi,
 \qquad p_H=\nu(A_H)>0.
 \label{102}
\end{equation}
Locality and bounded finiteness of \(\nu\) imply \(p_H<\infty\).

On \(\mathbb R^q\), the product boundedness consists of Borel sets contained in
\[
\{x\in\mathbb R^q:\min_{1\le i\le q}|x_i|>\varepsilon\}
\]
for some \(\varepsilon>0\). The corresponding test class consists of bounded continuous functions on \(\mathbb R^q\) that vanish whenever \(\min_{1\le i\le q}|x_i|\le\varepsilon\) for some \(\varepsilon>0\).

\begin{lemma}
\label{104} 
Under the assumptions on \(X\) and \(H\) we have on
\((\mathbb R\setminus\{0\})^q\)
\begin{equation}
 H(X)\in\RV_\Pi(\alpha,\nu_H;b),\qquad
 \nu_H=(H_\#\nu)|_{(\mathbb R\setminus\{0\})^q}.
 \label{105}
\end{equation}
If \(H\) takes values in \([0,\infty)^q\), then
\(p_H=\E\{\chi_\alpha(H(Z))\}\) and for every \(x\in\G\)
\begin{equation}
 \lim_{u\to\infty}a(u)
 \Pp\{H_i(X)>u_i x_i,\ 1\leq i\leq q\}
 =p_H\prod_{i=1}^q x_i^{-\alpha_i}.
 \label{108}
\end{equation}
\end{lemma}

Lemma~\ref{104} concerns a fixed map. Its conclusion cannot in general
be extended to an independent random homogeneous map by conditioning
and averaging, even when the random multiplier has moments of every
positive order.
\begin{example}
\label{117}

Let \(R_1,R_2\) be independent standard Pareto variables of index one
and let \(S\) be Pareto of index \(1/2\). Independently select
\(X=(R_1,R_2)\) or \(X=(S,1)\) with equal probabilities and regard
these vectors as constant paths. If \(\nu_0\) denotes the tensor-product
Pareto tail measure then \(X\in\RV_\Pi((1,1),\nu_0/2)\) because the
second branch eventually vanishes on every product-local test.

Independently take \(W=1+E\), where \(E\) is standard exponential,
and put \(\widetilde X=(X_1,WX_2)\). Although \(W\) has finite
moments of all positive orders, for \(u_1,u_2\ge1\)
\[
 u_1u_2\Pp\{\widetilde X_1>u_1,\widetilde X_2>u_2\}
 =1-\tfrac12e^{1-u_2}
   +\tfrac12u_1^{1/2}u_2e^{1-u_2}.
\]
Indeed, \(u\Pp\{WR_2>u\}=\E\{\min(W,u)\}=2-e^{1-u}\).
The display tends to one as \(u_1=u_2\to\infty\), whereas its
last term is \(ne^{n+1}/2\) at \((u_1,u_2)=(e^{4n},n)\) and
diverges as \(n\to\infty\). Thus multiplication destroys convergence
to a boundedly finite target with the
original power rates. The initially invisible second branch can become
visible at sufficiently unbalanced thresholds. Thus an independent
unbounded multiplier with finite moments of all positive orders need
not preserve MrRV at the original tail rates.
\end{example}

\subsection{Stationary realisation}
\label{225}

Proposition~\ref{38} supplies a Pareto realisation, which need not be
stationary. We begin with finite index sets, where the polar criterion
is global and stationary realisation on finite groups is immediate.

\begin{example}   
\label{ex:finite-polar}
If \(\mathsf T\) is finite, take \(K_n=\mathsf T\) for every \(n\ge1\).
Write \(M_i(f)=\max_{t\in\mathsf T}|f_i(t)|\) and
\((Sf)_i=f_i/M_i(f)\). Proposition~\ref{prop:local-polar} then
requires only one mass limit and one angular limit as \(u\to\infty\), with
\[
 c=\nu\{f:M_i(f)>1,\ 1\le i\le q\}\in(0,\infty),\qquad
 Q=c^{-1}S_\#\bigl(\1_{\{M_i>1,\,1\le i\le q\}}\nu\bigr).
\]
The law \(Q\) is supported by  
\(\{s:M_i(s)=1,\ 1\le i\le q\}\), and for every nonnegative
Borel \(F\),
\begin{equation}
 \nu(F)=c\int_\C\!\int_\G F(r\od s)\,\vartheta_\alpha(dr)\,Q(ds)
 \label{eq:finite-polar-factorisation}
\end{equation}
implying that  \(c,Q\) are uniquely determined by \(\nu\) for this
normalisation. Positivity of \(c\) follows because scaled copies of
\(\{M_i>1,\,1\le i\le q\}\) cover the carrier; the factorisation
follows by radial substitution in \eqref{14}.

If \(\mathsf T\) is a finite abelian group, the row maxima are
shift invariant and \(S\) commutes with shifts. Consequently,
\(\nu\) is shift-invariant if and only if \(Q\) is shift-invariant;
for a cyclic group it suffices to check its generator.
For strictly stationary \(X\), the conditional angular laws in the
criterion are already shift-invariant, hence so is their limit as \(u\to\infty\).
Alternatively, Corollary~\ref{50} uses only
\(|\mathsf T|^{q-1}\) relative-anchor checks.

For such a group and shift-invariant $\nu$, a stationary realisation
is obtained directly from Proposition~\ref{38}. Let $X_0$ be its
Pareto realisation and let $H$ be independent and uniform on
$\mathsf T$. Then $X=B^HX_0$ is strictly stationary and has tail
measure
\[
 \frac1{|\mathsf T|}\sum_{h\in\mathsf T}(B^h)_\#\nu=\nu.
\]
On infinite index sets, global suprema and the corresponding
exceedance mass need not be finite, so this global normalisation
is not automatic.
\end{example}

For the general case let $\mathsf T$ be a countable
discrete abelian group or $\mathbb R^m$ for a fixed $m\in\mathbb N$,
with counting or $m$-dimensional Lebesgue measure $\lambda$,
respectively. Fix a shift-invariant tail measure
$\nu$ in the sense of Definition~\ref{11}.
For $q\ge2$, the one-scale construction of
\cite[Theorem~3.7]{DombryHashorvaSoulier2018} does not apply directly:
product-window finiteness need not control one-component exceedances,
and different particles may supply different rows.
 
\begin{theorem}
\label{53}
With the power normalisers $b_i(u)=u^{\alpha_i}$, there exists
a strictly stationary $\C$-valued random field $X$ such that
\begin{equation}
 X\in\RV_\Pi(\alpha,\nu).
 \label{54}
\end{equation}
If $\nu$ is carried by nonnegative paths, $X$ can be chosen
nonnegative.
\end{theorem}

The proof is given in Appendix~\ref{226}.

\section{Applications and examples}
\label{61}

We apply the mapping principle to finite observation sets and compute
the full path-tail measure of a stationary finite moving average.

\subsection{Finite observation sets}

\begin{example}
\label{62}
Let \(\mathsf T=\{t_1,\ldots,t_k\}\), let \(X\) be nonnegative
with \(X\in\RV_\Pi(\alpha,\nu_Z;b)\), and choose a nonnegative
representer \(Z\). Write \(f_i^+(t)=\max\{f_i(t),0\}\).
For fixed \(c_{ij}>0\), \(1\le i\le q\), \(1\le j\le k\), define either the \emph{row maxima} or
\emph{weighted row aggregates} by
\[
 H_i(f)=\max_{1\le j\le k}f_i^+(t_j)
 \quad\text{or}\quad
 H_i(f)=\sum_{j=1}^k c_{ij}f_i^+(t_j),\qquad 1\le i\le q.
\]
These maps satisfy Lemma~\ref{104}, so for every \(x\in(0,\infty)^q\),
\begin{equation}
 \lim_{u\to\infty} a(u)
 \Pp\{H_i(X)>u_i x_i,\ 1\le i\le q\}
 =\E\left\{\prod_{i=1}^q H_i(Z)^{\alpha_i}\right\}
       \prod_{i=1}^q x_i^{-\alpha_i}.
 \label{63}
\end{equation}
The constant is finite and positive by \eqref{18} and
\(Z\in\C_\Pi^\circ\).

Row maxima need not correspond to a simultaneous exceedance.
For \(q=k=2\), let \(X_i(t_j)=R_i a_{ij}\), with independent
standard \(\alpha_i\)-Pareto radii, and compare the matrices
\[
 \begin{pmatrix}1&0\\0&1\end{pmatrix}
 \qquad\text{and}\qquad
 \begin{pmatrix}1&0\\1&0\end{pmatrix}.
\]
Both have row-maximum tail constant one; the tail constants for
\(\{\exists j\in\{1,2\}:X_i(t_j)>u_i,\ 1\le i\le2\}\), under
\(\chi_\alpha(u)\) as \(u\to\infty\), are respectively zero and one. Distinct row
anchors therefore matter already for two observations.
\end{example}

The next example discusses the simultaneous ruin at finitely many observation times. 

\begin{example}
Retain the setting of Example~\ref{62} and  for \(x\in\G\), define
the surplus \(U_i^{(u)}(t)=u_i x_i-X_i(t)\).
For every nonempty \(J\subset\{1,\ldots,k\}\) set next 
\[
 H_{J,i}(f)=\min_{j\in J}f_i^+(t_j),\qquad
 p_J=\E\left\{\prod_{i=1}^q H_{J,i}(Z)^{\alpha_i}\right\}.
\]
Each \(H_J\) is continuous, componentwise homogeneous and satisfies
the locality condition in \eqref{102}. Lemma~\ref{104} therefore
implies  whenever \(p_J>0\)
\[
 \lim_{u\to\infty}a(u)
 \Pp\{X_i(t_j)>u_i x_i,\ 1\le i\le q,\ j\in J\}
 =p_J\prod_{i=1}^q x_i^{-\alpha_i}.
\]
The same limit holds when \(p_J=0\): apply the lemma to
\(H_J+\varepsilon H^\vee\), where
\(H_i^\vee(f)=\max_{1\le j\le k}f_i^+(t_j)\), and let
\(\varepsilon\downarrow0\), using \eqref{18} and dominated convergence.

Inclusion--exclusion now yields the simultaneous ruin limit
\[
 \lim_{u\to\infty}a(u)
 \Pp\{\exists j\in\{1,\ldots,k\}:
       U_i^{(u)}(t_j)<0,\ 1\le i\le q\}
 =
 \left(
 \sum_{\varnothing\ne J\subset\{1,\ldots,k\}}
       (-1)^{|J|+1}p_J
 \right)
 \prod_{i=1}^q x_i^{-\alpha_i}.
\]
All lines are ruined at the same observation time.
\end{example}

\subsection{Finite moving averages}

A shared volatility path couples separately scaled innovation rows.
Their scalar path-tail measures can remain unchanged while their
relative-anchor tails differ. Related tail-process and finite-moving-average frameworks
are \cite{BasrakSegers2009,DombryHashorvaSoulier2018} and
\cite[Section~3.3 and Theorem~3.4]{ResnickRoy2014}, respectively.
For mixed volatility moments in joint heavy tails see
\cite[Section~5.2]{KulikSoulier2015}; for a positive-cone Breiman theorem see
\cite[Theorem~2.3]{JanssenDrees2016}.

Fix \(q=2\), \(\alpha>0\) and an integer \(\ell_0\ge2\).
Let \(V\) be a positive strictly stationary process on \(\mathbb Z\) satisfying
\begin{equation}
 \E\{V(0)^\alpha\}=1,\qquad \E\{V(0)^{2\alpha}\}<\infty.
 \label{86}
\end{equation}
Independently of $V$ let
$(\eta_{i,s})_{1\leq i\leq2;\,s\in\mathbb Z}$ be mutually independent
standard $\alpha$-Pareto variables and put
\begin{equation}
 X_i(t)=\sum_{j=0}^{\ell_0-1}V(t-j)\eta_{i,t-j},\qquad
 \varphi_s(t)=\1_{\{s,\ldots,s+\ell_0-1\}}(t),
 \label{87}
\end{equation}
\[w_\delta=\E\{V(0)^\alpha V(\delta)^\alpha\}, \qquad \delta\in\mathbb Z.\]

\begin{example}
\label{88}
The process $X$ is strictly stationary with
$X\in\RV_\Pi((\alpha,\alpha),\nu_{\mathrm{FV}};b)$ for $b_i(u)=u^\alpha$.
Its path-tail measure is
\begin{equation}
 \nu_{\mathrm{FV}}(F)
 =\sum_{\delta\in\mathbb Z}w_\delta
   \sum_{s\in\mathbb Z}\int_{(0,\infty)^2}
 F(r_1\varphi_s,r_2\varphi_{s+\delta})
 \,\vartheta_{(\alpha,\alpha)}(dr),\qquad F\ge0,
 \label{89}
\end{equation}
where $F$ is Borel measurable.

Only finitely many source pairs contribute to a fixed product window (recall definition \eqref{9}).
The scalar path-tail measure of either row, obtained from its separate
one-row regular-variation limit, is
\begin{equation}
 \nu_i^{\mathrm{sc}}(F)=\sum_{s\in\mathbb Z}\int_0^\infty
     F(r\varphi_s)\alpha r^{-\alpha-1}\,dr,\qquad 1\leq i\leq2,
 \label{90}
\end{equation}
so it does not depend on the temporal law of $V$. It is not a marginal
of \(\nu_{\mathrm{FV}}\): projecting the latter onto one row gives infinite
mass to every nonempty row-supremum window, because the unused radial
coordinate has infinite mass. The general anchor masses depend on the
lagged mixed volatility moments through
\begin{equation}
 p_{(0,\ell)}=\sum_{d=-(\ell_0-1)}^{\ell_0-1}(\ell_0-|d|)w_{\ell+d},
 \qquad\ell\in\mathbb Z.
 \label{91}
\end{equation}
These limits allow arbitrary relative growth of the two thresholds.
\end{example}

\begin{remark}
\label{rem:finite-filter-common-time}
The common-time calculation has a direct extension to any fixed
\(q\ge2\). Replace the two innovation rows in \eqref{87} by \(q\)
mutually independent rows, still independent of \(V\), and assume
\[
 \E\{V(0)^{q\alpha}\}<\infty.
\]
Put \(M_q(0)=0\) and define
\[
 M_q(m)
 =\E\left\{\left(\sum_{j=0}^{m-1}V(j)^\alpha\right)^q\right\}, \qquad m\in \mathbb N.
\]
For \(T\in\mathbb N\), let \(H_q(T)\) denote the limiting tail mass of
at least one simultaneous exceedance of all \(q\) rows over
\(\{0,\ldots,T-1\}\), i.e., 
\[
 H_q(T)
 =\lim_{u\to\infty}
   \left(\prod_{i=1}^q u_i^\alpha\right)
   \Pp\left\{\exists\,0\le t<T:
        X_i(t)>u_i,\ 1\le i\le q\right\}.
\]
Fix \(T\), put \(J=\{0,\ldots,T-1\}\) and
\(S_J=\bigcup_{j=0}^{\ell_0-1}(J-j)\). Conditional independence
and \eqref{156} give, for every \(u\in(0,\infty)^q\),
\[
 \begin{aligned}
 &\left(\prod_{i=1}^q u_i^\alpha\right)
 \Pp\{\exists t\in J:\ X_i(t)>u_i,\ 1\le i\le q\mid V\}\\
 &\quad\le\prod_{i=1}^q
 \left[u_i^\alpha\Pp\{\max_{t\in J}X_i(t)>u_i\mid V\}\right]\\
 &\quad\le\ell_0^{q\alpha}
       \left(\sum_{s\in S_J}V(s)^\alpha\right)^q=:D_J,
 \end{aligned}
\]
and
\[
 \E\{D_J\}
 \le\ell_0^{q\alpha}|S_J|^{q-1}
       \sum_{s\in S_J}\E\{V(s)^{q\alpha}\}
 =\ell_0^{q\alpha}|S_J|^q\E\{V(0)^{q\alpha}\}<\infty.
\]
The conditional scalar limits in Appendix~\ref{230} tensorise
across the \(q\) rows; the simultaneous-exceedance event has null
boundary under the limit by the radial densities. Dominated
convergence with \(D_J\) therefore gives \(H_q(T)\) along every
componentwise diverging threshold sequence, for each fixed \(T\).
Then
\[
 p_q:=H_q(1)=M_q(\ell_0),\qquad
 h_q=M_q(\ell_0)-M_q(\ell_0-1),
\]
and
\[
 H_q(T)=p_q+(T-1)h_q,
\]
so
\[
 \lim_{T\to\infty}\frac{H_q(T)}{T\,H_q(1)}
 =\theta_{q,\ell_0}
 :=\frac{h_q}{p_q}
 =1-\frac{M_q(\ell_0-1)}{M_q(\ell_0)}.
\]
For $r\in\mathbb Z_{\ge0}$, put
\[
 c_q(r)=
 \sum_{\substack{\boldsymbol d\in\{0,\ldots,r\}^q\\
                  \min_{1\le i\le q}d_i=0,\ \max_{1\le i\le q}d_i=r}}
 \E\left\{\prod_{i=1}^q V(d_i)^\alpha\right\}.
\]
Expanding the power and grouping source tuples by their minimum
and range gives, by stationarity,
\[
 \begin{aligned}
 M_q(m)
 &=\sum_{\boldsymbol s\in\{0,\ldots,m-1\}^q}
      \E\left\{\prod_{i=1}^qV(s_i)^\alpha\right\}
   =\sum_{r=0}^{m-1}(m-r)c_q(r),\\
 M_q(m)-M_q(m-1)&=\sum_{r=0}^{m-1}c_q(r),\qquad m\ge1.
 \end{aligned}
\]
For source times \(s_i=s+d_i\), with \(\min_{1\le i\le q}d_i=0\) and
\(\max_{1\le i\le q}d_i=r\), the common pulse support is
\[
 \bigcap_{i=1}^q\{s_i,\ldots,s_i+\ell_0-1\}
 =[s+r,s+\ell_0-1]\cap\mathbb Z.
\]
It meets \(J\) precisely when
\[
 0\le r<\ell_0,\qquad 1-\ell_0\le s\le T-1-r.
\]
Each such source tuple contributes
\(\E\{\prod_{i=1}^qV(s_i)^\alpha\}\), since the unit-threshold
radial mass is
\(\prod_{i=1}^q\int_1^\infty\alpha r_i^{-\alpha-1}\,dr_i=1\).
Consequently,
\[
 \begin{aligned}
 H_q(T)
 &=\sum_{r=0}^{\ell_0-1}
     \sum_{s=1-\ell_0}^{T-1-r}c_q(r)=\sum_{r=0}^{\ell_0-1}(T+\ell_0-r-1)c_q(r)\\
 &=M_q(\ell_0)
   +(T-1)\bigl[M_q(\ell_0)-M_q(\ell_0-1)\bigr].
 \end{aligned}
\]
For $q=2$, $p_q$ equals $p_{(0,0)}$ in \eqref{91}.

Thus the dependence on the codomain dimension enters through
higher-order mixed moments of \(V^\alpha\), while \(T\) records growth
of the observation window.
\end{remark}

\begin{example}
\label{97}
Take any $\ell_0\ge4$ and $0<c<1/7$, and let $(\xi_t)$ be iid Rademacher
variables independent of the innovations. Compare
\begin{equation}
 \begin{aligned}
 V_A(t)^\alpha&=1+c(2\xi_{t-1}+\xi_{t-2}+4\xi_{t-3}),\\
 V_B(t)^\alpha&=1+c(\xi_t+4\xi_{t-2}+2\xi_{t-3}).
 \end{aligned}
 \label{98}
\end{equation}
Both volatility processes are positive, bounded and strictly stationary,
and satisfy \eqref{86}. The coefficient permutations preserve their
one-time law, and \eqref{90} gives the same scalar path-tail measure
for every row. Nevertheless, \eqref{91} yields
\begin{equation}
 \begin{gathered}
 p_{(0,0)}^A=p_{(0,0)}^B=\ell_0^2+(49\ell_0-44)c^2,\\
 p_{(0,2)}^A-p_{(0,2)}^B
 =\begin{cases}
 2c^2,&\ell_0=4,\\
 4c^2,&\ell_0\ge5.
 \end{cases}
 \end{gathered}
 \label{99}
\end{equation}
Moreover, Remark~\ref{rem:finite-filter-common-time} gives
\[
 \begin{aligned}
 H_2^A(T)=H_2^B(T)
 &=\ell_0^2+(49\ell_0-44)c^2 +(T-1)(2\ell_0-1+49c^2),\qquad T\in\mathbb N.
 \end{aligned}
\]
Thus even the scalar path-tail measures together with the
simultaneous-exceedance tail masses on every consecutive finite window
do not determine the relative-anchor tails.
\end{example}
The proofs of \eqref{89}--\eqref{91} and \eqref{99}
are presented in Appendix~\ref{230}.
\section{Proofs}
\label{101}

\begin{proof}[Proof of Proposition~\ref{38}]
Use the countable dense subset \(\mathsf D\) from \eqref{184} and enumerate
\(\mathsf D^q\) as \((\boldsymbol h_j)_{j\in\mathcal J}\), where
\(\mathcal J=\{1,\ldots,N\}\) with \(N=|\mathsf D^q|\) when
\(\mathsf D^q\) is finite and \(\mathcal J=\mathbb N\) otherwise. By continuity of the paths
\(\C_\Pi^\circ=\bigcup_{j\in\mathcal J}E_{\boldsymbol h_j}\) and we disjointify
this cover into the Borel sets
\[
 V_j=E_{\boldsymbol h_j}\setminus
       \bigcup_{\substack{k\in\mathcal J\\k<j}}E_{\boldsymbol h_k}.
\]
Each \(V_j\) is invariant under componentwise scaling. For \(j\in\mathcal J\)
and \(1\le i\le q\), put on \(V_j\)
\[
 a_{j,i}(f)=|f_i(h_{j,i})|,\qquad
 s_j(f)=a_j(f)^{-1}\od f,\qquad
 \mathcal S_j=\{f\in V_j:a_j(f)=\boldsymbol1\}.
\]
The map \(f\mapsto(a_j(f),s_j(f))\) is a Borel isomorphism
from \(V_j\) onto \(\G\times \mathcal S_j\) with inverse
\((r,s)\mapsto r\od s\) and we define the finite measure
\begin{gather*}
 \sigma_j(C)=\nu\{f\in V_j:a_{j,i}(f)>1\ (1\leq i\leq q),\
                              s_j(f)\in C\},\\
 C\in\mathcal B(\mathcal S_j),
 \qquad m_j=\sigma_j(\mathcal S_j).
\end{gather*}
Finiteness follows because the set defining \(m_j\) is contained
in a product window \eqref{9}, whose mass is finite by \eqref{12b}.
By scale invariance of \(V_j\) and
\(s_j(r\od f)=s_j(f)\) multihomogeneity implies
\[
 \nu\{f\in V_j:a_{j,i}(f)>x_i\ (1\leq i\leq q),\ s_j(f)\in C\}
 =\chi_\alpha(x)^{-1}\sigma_j(C),\qquad x\in\G.
\]
Upper radial rectangles times Borel subsets of \(\mathcal S_j\) determine
the finite measures on \((1/n,\infty)^q\times \mathcal S_j\) so letting
\(n\to\infty\) identifies the pushforward of \(\nu|_{V_j}\) as
\(\vartheta_\alpha\otimes\sigma_j\). For every Borel map \(F\geq0\)
we therefore have
\begin{equation}
 \nu(F)=\sum_{j\in\mathcal J}\int_{\mathcal S_j}\int_\G
          F(r\od s)\,\vartheta_\alpha(dr)\,\sigma_j(ds).
 \label{130}
\end{equation}
This also shows that sectors with \(m_j=0\) are \(\nu\)-null.

Since \(\nu\ne0\) the set \(J_+=\{j\in\mathcal J:m_j>0\}\) is nonempty.
Choose positive probabilities \((\pi_j)_{j\in J_+}\) and let \(J\)
have this distribution with conditional law \(\sigma_j/m_j\) for
\(S=(S_1,\ldots,S_q)\) on \(\{J=j\}\). Set
\[
 Z_1=(m_J/\pi_J)^{1/\alpha_1}S_1,\qquad
 Z_i=S_i, \quad 2\leq i\leq q.
\]
This defines a random element of \(\C_\Pi^\circ\) and changing
the first radial variable in \eqref{14}
multiplies its radial integral by \(m_j/\pi_j\) on \(\{J=j\}\).
In the light of \eqref{130} we obtain
\(\nu_Z=\nu\) while direct radial integration on each product window in \eqref{9} yields
\[
 \E\left\{\prod_{i=1}^q\rho_{n_i,i}(Z)^{\alpha_i}\right\}
 =\nu(U_{\boldsymbol n,\boldsymbol1})<\infty,
\]
as required without any finiteness assumption on \(\sum_{j\in\mathcal J}m_j\).
The construction preserves every scaling-invariant Borel carrier
because it only applies componentwise positive rescalings to the
original shapes.

For the realisation we change variables in the Pareto densities to obtain
\[
 \chi_\alpha(u)\Pp\{u^{-1}\od X\in A\}
 =\E\left\{\int_{\{r_i>1/u_i,\ 1\leq i\leq q\}}
       \1_{\{r\od Z\in A\}}\,\vartheta_\alpha(dr)\right\}
\]
for every Borel \(A\subset\C_\Pi^\circ\). These measures are
dominated by \(\nu_Z\) and increase to it as \(u\to\infty\).
On a finite-mass product window the total variation distance of the
restrictions is the difference of their masses so monotone convergence
along diagonal thresholds and domination by every sufficiently large
threshold vector prove the full-net conclusion.

\end{proof}

\begin{proof}[Proof of Proposition~\ref{prop:local-polar}]
Fix \(n\in\mathbb N\) and write \(g_i=\rho_{n,i}\), \(1\le i\le q\), \(B=B_n\), \(S=S_n\),
\(c=c_n\). On \(E=\{f:g_i(f)>0,\ 1\le i\le q\}\), radial
substitution in \eqref{14}, using
\(g_i(r\od f)=r_i g_i(f)\) yields when \(c>0\)
\[
 \nu(F\1_E)=c\int_\C\!\int_\G F(r\od s)\,
                  \vartheta_\alpha(dr)\,Q_n(ds),\qquad F\ge0.
\]
Indeed the angular mixing measure is the image under \(S\) of
\(\prod_{i=1}^q g_i(Z)^{\alpha_i}\,d\Pp\), with total mass \(c\);
restriction to \(B\) identifies its normalisation as \(Q_n\).
In particular the joint radial-angular law under
\(c^{-1}\1_B\nu\) is \(\Law(R)\otimes Q_n\).
If MrRV holds, Lemma~\ref{133} and \eqref{25} give
\(\1_B\mu_u\Longrightarrow\1_B\nu\) as \(u\to\infty\).
The mass limit follows, and continuous mapping on \(E\) gives the
angular and joint limits when \(c>0\).

Conversely suppose the stated mass and angular limits hold.
For \(c>0\) set \[  E_u=\{g_i(X)>u_i,\ 1\le i\le q\}.\]
For every \(x\in[1,\infty)^q\) and bounded continuous \(H\) on
\(\C\), scale invariance of \(S\) and the assumptions at
\(x\od u\) imply as $u\to \infty$
\[
 \E\left\{H(SX)\1_{\{g_i(X)>x_iu_i,\,1\le i\le q\}}
             \mid E_u\right\}
 \longrightarrow
 \left(\prod_{i=1}^q x_i^{-\alpha_i}\right)\int_\C H\,dQ_n.
\]
Here \(\lim_{u\to\infty}a(u)/a(x\od u)=\prod_{i=1}^q x_i^{-\alpha_i}\).
Taking \(H=1\) shows convergence of the radial marginals to Pareto
laws, hence their tightness. The assumed angular convergence gives
joint tightness, and the displayed identities identify every joint
subsequential limit as \(\Law(R)\otimes Q_n\).
Multiplication \((r,s)\mapsto r\od s\) is continuous, so
\(\1_{B_n}\mu_u\Longrightarrow\1_{B_n}\nu\) as \(u\to\infty\).
For \(c_n=0\) this follows directly from the mass limit.
Applying this conclusion at thresholds \(u/n\) and scaling paths
by \(1/n\) gives convergence on \(O_n=(1/n)\od B_n\) as \(u\to\infty\), since
\(\lim_{u\to\infty}a(u)/a(u/n)=n^{\sum_{i=1}^q\alpha_i}\).
Lemma~\ref{133} now proves MrRV.

For the anchored formulation use \(g_i(f)=|f_i(h_i)|\) in the
same argument. Equations~\eqref{20}--\eqref{23} identify the
angular law as \(\Law(\Theta^{[\boldsymbol h]})\).
This proves the equivalence of \eqref{48} and
\eqref{eq:anchored-polar-angle} under \eqref{47};
Theorem~\ref{44} supplies the asserted index-set distinctions.
\end{proof}

\begin{proof}[Proof of Theorem~\ref{44}]
Necessity follows from Lemma~\ref{133} and
\eqref{43}. For sufficiency we first transfer
\eqref{45} to every fixed
\(x\in\G\) by writing \(v=x\od u\) and observing that
as \(u\to\infty\)
\begin{equation}
 \1_{A_{\boldsymbol h,x}}\mu_u
 =\frac{a(u)}{a(x\od u)}
 (f\mapsto x\od f)_\#
 \bigl(\1_{A_{\boldsymbol h}}\mu_{x\od u}\bigr)
 \Longrightarrow\1_{A_{\boldsymbol h,x}}\nu.
 \label{118}
\end{equation}
Here \(\lim_{u\to\infty}a(u)/a(x\od u)=\chi_\alpha(x)^{-1}\) by fixed-level
regular variation of the \(b_i\)'s and multihomogeneity identifies
the limit.

Fix \(n\in\mathbb N\). By density of \(\mathsf D\cap K_n\) and path continuity,
the anchor sets \(A_{\boldsymbol h,(1/n,\ldots,1/n)}\),
\(\boldsymbol h\in(\mathsf D\cap K_n)^q\), form a countable cover
of \(O_n\), finite in the discrete case. Enumerate these sets as
\((A_j)_{j\ge1}\), repeating sets in the finite case, put \(V_m=\bigcup_{j=1}^m A_j\), and disjointify by
\(D_j=A_j\setminus\bigcup_{1\le k<j}A_k\). Then \(V_m\uparrow O_n\)
as \(m\to\infty\), with equality eventually in the discrete case. Since
\(\partial D_j\subset\bigcup_{k=1}^j\partial A_k\), \eqref{43}
gives \(\nu(\partial D_j)=0\). Restricting \eqref{118} for \(A_j\)
to \(D_j\) and summing over \(1\le j\le m\) therefore give
\[
 \1_{V_m}\mu_u\Longrightarrow\1_{V_m}\nu,\qquad u\to\infty.
\]
In the continuous case, \eqref{46} yields
\[
 \limsup_{u\to\infty}\mu_u(O_n\setminus V_m)
 \leq\nu(O_n)-\nu(V_m)\longrightarrow0,
 \qquad m\to\infty;
\]
in the discrete case the remainder is eventually zero.
Approximation by the restrictions to \(V_m\) therefore proves
convergence on \(O_n\) as \(u\to\infty\), so
Lemma~\ref{133} applies. The assertions about zero and positive
anchor masses follow by taking total masses and normalising finite
measures. The extension to arbitrary anchor tuples follows from the
full convergence just proved and restriction to the corresponding
continuity sets.
\end{proof}

\begin{proof}[Proof of Corollary~\ref{50}]
Necessity follows from Theorem~\ref{44} and shift invariance of the
tail measure of a stationary MrRV rf. For the converse, take \(\boldsymbol h\in\mathsf D^q\), put
\(\ell_i=h_i-h_1\ (2\leq i\leq q)\), and use stationarity of \(X\) and shift invariance
of \(\nu\) to obtain
\[
 \1_{A_{\boldsymbol h}}\mu_u
 =(B^{h_1})_\#(\1_{A_\ell}\mu_u),\qquad
 \1_{A_{\boldsymbol h}}\nu
 =(B^{h_1})_\#(\1_{A_\ell}\nu).
\]
Continuity of the shift transfers \eqref{51} to every
absolute-anchor condition \eqref{45}; when required, \eqref{46}
completes the proof. Moreover,
\(B^aA_{\boldsymbol h}=A_{\boldsymbol h+a}\) and
\(N_{\boldsymbol h+a}B^a=B^aN_{\boldsymbol h}\), which prove
\eqref{52}.
\end{proof}

\begin{proof}[Proof of Lemma~\ref{104}]
By equivariance we have for every \(\varepsilon\in\G\)
\[
 H^{-1}\{y:|y_i|>\varepsilon_i,\ 1\leq i\leq q\}
 =\varepsilon\od A_H\in\mathcal B_\Pi.
\]
A bounded continuous target test vanishing near the deleted coordinate
hyperplanes therefore pulls back to a bounded product-local function
whose discontinuity set is \(\nu\)-null. After choosing a cutoff from
\eqref{132} equal to one on its support we apply
Lemma~\ref{133} and the ordinary mapping theorem
for the resulting finite measures to obtain convergence of its integrals as \(u\to\infty\).
This proves \eqref{105} for a boundedly
finite target which is nonzero by \(p_H>0\).

For nonnegative \(H\) we substitute \(s_i=r_iH_i(Z)\) in
\eqref{14}. Outcomes with a zero coordinate of
\(H(Z)\) do not contribute on the target carrier and the substitution
identifies \(\nu_H=p_H\vartheta_\alpha\) and
\(p_H=\E\{\chi_\alpha(H(Z))\}\).
By the same radial integration \(\nu\{H_i=c\}=0\) for \(c>0\) and \(1\le i\le q\)
so continuity of \(H\) \(\nu\)-almost everywhere implies
\(\nu(\partial A_H)=0\). Restriction convergence from
Lemma~\ref{133} yields the mass limit in
\eqref{108} and scaling treats every fixed \(x\).
\end{proof}

\appendix
\section{Technical details}
\label{129}

\subsection{Convergence and localisation}

Use the notation of Sections~\ref{6} and~\ref{224}. The
localisation arguments are standard finite-measure arguments applied
to the product windows; see
\cite[Corollary~B.1.19]{KulikSoulier2020} and
\cite[Remark~4.4(i)]{BladtHashorvaShevchenko2022} for the local
restriction principle. All threshold limits are taken as
\(u\to\infty\) in the sense of \eqref{30}.

\subsubsection*{Product localisers}

The metric \(d_0\) in \eqref{190} induces locally uniform convergence:
convergence in \(d_0\) implies convergence in each compact-window
seminorm, and the converse follows by truncating the summable series.
Its triangle inequality follows from that of the seminorms and
subadditivity of \(s\mapsto1\wedge s\); invariance under addition is
immediate. For any compatible metric invariant under addition,
\[
 \delta_d(f)=\min_{1\le i\le q}
       \inf\{d(h,0):h\in\C,\ h_i=f_i\}.
\]
Indeed, distance to the finite union \(D_\Pi\) is the minimum of the
distances to its zero-component faces, and \(h=f-g\) transforms
\(g_i\equiv0\) into \(h_i=f_i\). For \(d_0\), erasing component
\(i\) while keeping the other components of \(f\) attains the
distance to the \(i\)-th face, giving
\begin{equation}
 \delta_{d_0}(f)
 =\min_{1\le i\le q}\sum_{n=1}^\infty
       2^{-n}\bigl(1\wedge\rho_{n,i}(f)\bigr).
 \label{191}
\end{equation}
Two compatible metrics invariant under addition are uniformly
equivalent: continuity of the identity at zero compares their balls
at zero, and invariance gives the same comparison for every pair of
paths. These comparisons preserve positive distance between sets,
so \eqref{182} is independent of the chosen admissible metric.

A product window is contained in \(\{\tau_N>\varepsilon\}\) for
some \(N\) and \(\varepsilon>0\). Equation~\eqref{191} then gives
the uniform lower bound \(2^{-N}(1\wedge\varepsilon)\) on its
distance from \(D_\Pi\). Conversely, suppose
\(\delta_{d_0}(f)\ge\eta>0\) for all \(f\in A\), and choose
\(N\) with \(2^{-N}<\eta/2\). Monotonicity of the gauges gives,
for every \(1\le i\le q\),
\[
 \eta\le\sum_{n=1}^\infty 2^{-n}(1\wedge\rho_{n,i}(f))
 \le\rho_{N,i}(f)+2^{-N}.
\]
Hence \(\tau_N(f)>\eta/2\) throughout \(A\). This proves the
equivalences \eqref{192}--\eqref{28}; the empty set is immediate.
They also show that \(\mathcal B_\Pi\) is hereditary and closed
under finite unions.

The cover \eqref{10} follows because every nonzero component has
positive supremum on some compact window. By \eqref{28}, every
\(A\in\mathcal B_\Pi\) lies in \(O_n\) for all sufficiently large
\(n\), while continuity and monotonicity of the moduli give
\[
 \overline O_n\subset\{\tau_n\ge1/n\}\subset O_{n+1}.
\]
Finally,
\[
 (\min_{1\le i\le q}r_i)\tau_n(f)\le\tau_n(r\od f)
 \le(\max_{1\le i\le q}r_i)\tau_n(f)
\]
proves \eqref{131}.

\subsubsection*{Localisation and multi-anchor convergence}

Fix a continuous non-decreasing function
\(\varphi:[0,\infty)\to[0,1]\) which is zero on \([0,1/2]\) and one
on \([1,\infty)\) and set
\begin{equation}
 \psi_{\boldsymbol n,\varepsilon}(f)
 =\prod_{i=1}^q
 \varphi\left(\frac{\rho_{n_i,i}(f)}{\varepsilon_i}\right).
 \label{132}
\end{equation}

Let \((\zeta_u)_{u\in\G}\) be a family of boundedly finite measures.

\begin{lemma}
\label{133}
The following conditions, all as \(u\to\infty\), are equivalent:
\begin{enumerate}
\item \(\zeta_u\xrightarrow{v_\Pi}\nu\);
\item for every \(\boldsymbol n\in\mathbb N^q\) and \(\varepsilon\in\G\),
 \begin{equation}
 \psi_{\boldsymbol n,\varepsilon}\zeta_u
 \Longrightarrow\psi_{\boldsymbol n,\varepsilon}\nu
 \quad\text{as finite measures on }\C;
 \label{134}
 \end{equation}
\item \(\1_{O_n}\zeta_u\Longrightarrow\1_{O_n}\nu\) as finite
 measures on \(\C\), for every \(n\in\mathbb N\).
\end{enumerate}
\end{lemma}

Under these equivalent conditions
\(\1_A\zeta_u\Longrightarrow\1_A\nu\) as \(u\to\infty\), for every
\(A\in\mathcal B_\Pi\) with \(\nu(\partial A)=0\) as shown in
the proof below.

\begin{proof}[Proof of Lemma~\ref{133}]
For bounded continuous \(G\) the function
\(G\psi_{\boldsymbol n,\varepsilon}\) belongs to \(C_\Pi(\C)\)
which proves 1\(\Rightarrow\)2 while
\(F=F\psi_{\boldsymbol n,\varepsilon}\) for every
\(F\in C_\Pi(\C)\) vanishing outside
\(U_{\boldsymbol n,\varepsilon}\) proves the converse.
The same cutoff equals one on any
\(A\subset U_{\boldsymbol n,\varepsilon}\).  Restricting the finite
weak convergence in 2 to the continuity set \(A\) proves the
assertion following the lemma and hence 3 by
\eqref{43}. Finally every product-local test
function vanishes outside some \(O_n\) by
\eqref{28} so 3 implies 1.
\end{proof}

\subsubsection*{Finite-dimensional and equicontinuity checks}

Let \(\eta_u,\eta\) be finite Borel measures on \(\C\) and write
\(\pi_Jf=(f(t))_{t\in J}\) for finite \(J\subset\mathsf D\).
In the continuous case put
\[
 \omega_{s,i}(f,\delta)
 =\sup_{\substack{x,y\in K_s\\|x-y|\leq\delta}}
 |f_i(x)-f_i(y)|.
\]

\begin{lemma}
\label{135}
As \(u\to\infty\), on a discrete index set
\(\eta_u\Longrightarrow\eta\) is equivalent to
\((\pi_J)_\#\eta_u\Longrightarrow(\pi_J)_\#\eta\) for every finite
\(J\subset\mathsf D\), including convergence of total masses.
On \(\mathbb R^m\) the same equivalence holds after adding for every
\(s\in\mathbb N\), \(1\leq i\leq q\) and \(\gamma>0\)
\[
 \lim_{\delta\downarrow0}\limsup_{u\to\infty}
 \eta_u\{f\in\C:\omega_{s,i}(f,\delta)>\gamma\}=0.
\]
\end{lemma}

The lemma applies to the finite cutoff measures in
Lemma~\ref{133} and the anchor restrictions in
Theorem~\ref{44}.

\begin{proof}[Proof of Lemma~\ref{135}]
Work along an arbitrary sequence \((u^{(n)})\) with
\(u^{(n)}\to\infty\) as \(n\to\infty\).
In the discrete case coordinate tightness with summable error bounds
implies tightness on the countable product. In the continuous case a
finite \(\delta\)-net \(J\subset\mathsf D\cap K_s\) yields
\[
 \rho_{s,i}(f)\leq\max_{t\in J}|f_i(t)|
 +\omega_{s,i}(f,\delta).
\]
Finite-dimensional convergence and the modulus condition therefore
imply tightness of compact suprema so Arzel\`a--Ascoli on successive
boxes establishes path-space tightness while evaluations on \(\mathsf D\)
identify every subsequential limit. For the converse continuous mapping
establishes the finite-dimensional limits and Portmanteau bounds the limsup in the
modulus condition by
\(\eta\{\omega_{s,i}(f,\delta)\ge\gamma\}\), which tends to zero
as \(\delta\downarrow0\) by continuity of paths and finiteness of \(\eta\).
Since the sequence was arbitrary, the conclusion holds as \(u\to\infty\).
\end{proof}

\subsection{Stationary realisation}
\label{226}

\begin{lemma}
\label{lem:sr-exhaustion}
Every countably infinite discrete additive abelian group $\mathsf T$ admits an increasing
finite exhaustion $(K_n)$, with $0\in K_n$, such that as $n\to \infty$
\[
 \frac{|K_{n+1}|}{|K_n|}\longrightarrow1,\qquad
 \frac{|(h+K_n)\mathbin\triangle K_n|}{|K_n|}\longrightarrow0,
 \quad h\in\mathsf T.
\]
\end{lemma}

\begin{proof}
Enumerate the group and choose finite sets $A_m$ whose relative
translation errors are at most $1/m$ for its first $m$ elements.
Such sets are obtained from large boxes in the finitely generated
abelian subgroup generated by those elements; a finite subgroup
itself suffices when that subgroup is finite.

Start with a finite set containing $0$. At stage $m$, add the $m$th
enumerated element if absent, followed by disjoint translates of
$A_m$. End the stage only after the accumulated size is at least
$m$ times its initial size, at least $m|A_{m+1}|$, and at least $m^2$.
An infinite group always permits another translate disjoint from the
current finite union. Retain every intermediate set in the exhaustion.
Immediately before an addition at stage $m\ge2$, the current set
$K$ and the enlarged set $K'$ satisfy
\[
 |K|\ge\max\{(m-1)|A_m|,(m-1)^2\},\qquad
 0\le\frac{|K'|}{|K|}-1
 \le\frac{\max\{|A_m|,1\}}{|K|}\le\frac1{m-1}
\]
implying $\lim_{n\to \infty}|K_{n+1}|/|K_n|=1$.
Fix $h\in\mathsf T$, choose $r\ge2$ so that $h$ is among the
first $r$ enumerated elements, and let $E_r$ be the set at the end
of stage $r-1$. Every intermediate set $K$ in stage $m\ge r$
is the disjoint union of $E_r$, translates of blocks $A_j$ with
$j\ge r$, and at most $m$ further singletons. Hence we have 
\[
 \frac{|(h+K)\mathbin\triangle K|}{|K|}
 \le\frac{2|E_r|}{|K|}+\frac1r+\frac{2m}{|K|}
 \le\frac{2|E_r|+2m}{(m-1)^2}+\frac1r.
\]
Letting $m\to\infty$ and then $r\to\infty$ proves the second
assertion. Enumeration ensures exhaustion.
\end{proof}

\begin{proof}[Proof of Theorem~\ref{53}]
The finite-group case follows from Example~\ref{ex:finite-polar}.
Suppose henceforth that $\mathsf T$ is infinite. In the discrete
case use the exhaustion of Lemma~\ref{lem:sr-exhaustion}; on
$\mathbb R^m$ use $K_n=[-n,n]^m$. By cofinality these exhaustions may
be used in \eqref{27}. Write
$\pi_n f=f|_{K_n}$, $\ell_n=\lambda(K_n)$ and
$\beta=\sum_{i=1}^q\alpha_i$. Choose a representer $Z$ with $\nu=\nu_Z$
by Proposition~\ref{38}, taking $Z$ nonnegative when appropriate.
 
In the discrete case put $J_n=K_n$ and $\kappa_n=\1_{K_n}$.
On $\mathbb R^m$ put $J_n=[-n-1,n+1]^m$ and choose continuous
$0\le\kappa_n\le1$, equal to one on $K_n$ and zero outside $J_n$.
For the stationary placement below define
\[
 d_n=
 \begin{cases}
 |(K_n-K_n)+(K_n-K_n)|,&\mathsf T\text{ discrete},\\
 [8(n+1)]^m,&\mathsf T=\mathbb R^m.
 \end{cases}
\]
These numbers are positive and nondecreasing. For $1\le j\le n$ define
the finite measure
\[
 d\Lambda_{n,j}
 =\1_{\{s\in K_j-J_n\}}\1_{\{B^s(v\od Z)\in O_j\}}
       \lambda(ds)\,\Pp_Z(dZ)\,\vartheta_\alpha(dv).
\]
Its mass is $\lambda(K_j-J_n)\nu(O_j)$ by shift invariance.
Choose $\delta_n\in(0,1)$ and $v_n^+>1$ so that, simultaneously
for $1\le j\le n$,
\begin{equation}
 \Lambda_{n,j}\{\min_{1\le i\le q} v_i<\delta_n\}\le n^{-3},\qquad
 \Lambda_{n,j}\{\max_{1\le i\le q} v_i>v_n^+\}\le2^{-n}.
 \label{sr:radial-errors}
\end{equation}
Choose $t_1\ge2$ and $t_n$ recursively large enough that
\begin{equation}
 \frac{t_n}{t_{n-1}}>\max\{2,\delta_n^{-1},v_n^+\},\quad n\ge2,
 \qquad \frac{d_n}{\ell_n}t_n^{-\beta}\le2^{-n},\quad n\ge1.
 \label{sr:tiers}
\end{equation}
Let $n(r)=n$ when $t_n\le\min_{1\le i\le q} r_i<t_{n+1}$, and define
the finite mark measure and its associated path by
\begin{align*}
 d\mathsf M(z)&=
 \frac{\1_{\{\min_{1\le i\le q} r_i\ge t_1\}}}{\ell_{n(r)}}
       \vartheta_\alpha(dr)\,\Pp_Z(dZ),\qquad z=(r,Z),\\
 P_z&=\kappa_{n(r)}(r\od Z).
\end{align*}
Then $0<\mathsf M(1)<\infty$ and
\begin{equation}
 \int d_{n(r)}\,\mathsf M(dz)
 \le\sum_{n=1}^\infty\frac{d_n}{\ell_n}t_n^{-\beta}<\infty.
 \label{sr:base-cost}
\end{equation}
For localised tests put
\[
 \lambda_u(\Phi)=\chi_\alpha(u)
 \int\int_{\mathsf T}\Phi(u^{-1}\od B^sP_z)
                      \lambda(ds)\,\mathsf M(dz).
\]
For each fixed $j$, the ratio $\lambda(K_j-J_n)/\ell_n$ is
bounded uniformly in $n$: it is at most $|K_j|$ in the discrete
case and $(j+2)^m$ on $\mathbb R^m$. For a bounded test $\Phi$ supported
on $O_j$, summing over the disjoint radial tiers gives
\[
 \lambda_u(|\Phi|)\le\chi_\alpha(u)\|\Phi\|_\infty
 \left(\sup_{n\ge1}\frac{\lambda(K_j-J_n)}{\ell_n}\right)t_1^{-\beta}
 <\infty.
\]
For integers $1\le N\le j$ and every bounded continuous
$F:C(K_j,\mathbb R^q)\to\mathbb R$, we claim that, as $u\to\infty$
\begin{equation}
 \lambda_u(\Phi)\longrightarrow\nu(\Phi),\qquad
 \Phi=\1_{O_N}(F\circ\pi_j).
 \label{sr:raw-limit}
\end{equation}

Assume $\|F\|_\infty\le1$, so $|\Phi|\le\1_{O_j}$.
Substitute $r=u\od v$ and let
$t_M\le\min_{1\le i\le q} u_i<t_{M+1}$, where $M\to\infty$ as $u\to\infty$,
and assume $M>j+2$.
The selected tier depends on $a=\min_{1\le i\le q} u_iv_i$.
For $1\le n\le M-2$, one has
$\min_{1\le i\le q} v_i<t_{M-1}/t_M<\delta_M$.
Contributing centres lie in $K_j-J_n\subset K_j-J_M$, and
membership of the clipped path in $O_j$ implies membership of
the unclipped path. These tiers therefore contribute at most
$M^{-3}\sum_{n=1}^{M-2}\ell_n^{-1}\le M^{-2}$.
For $n\ge M+2$, choose $i_*\in\{1,\ldots,q\}$ with $u_{i_*}=\min_{1\le i\le q} u_i$. Then
\[
 v_{i_*}\ge t_n/u_{i_*}>t_n/t_{M+1}
 \ge t_n/t_{n-1}>v_n^+,
\]
so the high tiers contribute at most
$\sum_{n=M+2}^\infty\ell_n^{-1}2^{-n}=O(2^{-M})$ as $M\to\infty$.

For $n=M-1,M,M+1$, all cutoffs equal one throughout the translated
observation window when
\[
 s\in I_M:=\bigcap_{t\in K_j}(t-K_{M-1}).
\]
The exhaustion lemma, applied also to $-K_n$, gives in the discrete
case, as $M\to\infty$,
\[
 \frac{\lambda(I_M)}{\ell_M}\to1,\qquad
 \frac{\lambda((K_j-J_{M+1})\setminus I_M)}{\ell_M}\to0,
 \qquad \frac{\ell_{M\pm1}}{\ell_M}\to1.
\]
On $\mathbb R^m$ the same assertions follow directly from
$I_M=[-M+j+1,M-j-1]^m$ and $\ell_M=(2M)^m$.
Replacing the three denominators by
$\ell_M$ has error $o(1)\nu(O_j)$ as $M\to\infty$. Their tier indicators sum
to one except on
$\{a<t_{M-1}\}\cup\{a\ge t_{M+2}\}$.
The first set forces $\min_{1\le i\le q} v_i<\delta_M$; the second forces
$\max_{1\le i\le q} v_i>v_{M+2}^+$. Since $I_M$ is contained in the respective
centre domains, \eqref{sr:radial-errors} bounds the missing terms.
Restoring the full radial integral and using shift invariance yields
as $M\to\infty$
\[
 \frac1{\ell_M}\int_{I_M}
 \E\Bigl\{\int_\G\Phi(B^s(v\od Z))\,\vartheta_\alpha(dv)\Bigr\}\,\lambda(ds)
 =\frac{\lambda(I_M)}{\ell_M}\nu(\Phi)\longrightarrow\nu(\Phi).
\]
The remaining centres contribute at most
\[
 \frac{3\lambda((K_j-J_{M+1})\setminus I_M)}{\ell_{M-1}}
       \nu(O_j)=o(1),\qquad M\to\infty.
\]
This proves \eqref{sr:raw-limit} without restricting $u_i/u_k$, $1\le i,k\le q$.
Shift invariance is applied only after restoring the full radial
integral.

 Each $P_z$ has compact support and hence  its amplitude
\[
 S_z=\max_{1\le i\le q}\sup_{t\in J_{n(r)}}|P_{z,i}(t)|
\]
is finite and measurable. After $\mathsf M$ has been fixed, choose
strictly increasing positive $s_j\uparrow\infty$ as $j\to\infty$ such that
\[
 d_{j+1}\mathsf M\{S_z>s_j\}\le2^{-j},\qquad j\ge1.
\]
Put $L(S)=\min\{j\ge1:S\le s_j\}$ and
$k(z)=\max\{n(r),L(S_z)\}$. Since
\[
 d_{k(z)}\le d_{n(r)}+
       \sum_{j=1}^\infty d_{j+1}\1_{\{S_z>s_j\}},
\]
\eqref{sr:base-cost} implies
\begin{equation}
 0<K:=\int d_{k(z)}\,\mathsf M(dz)<\infty.
 \label{sr:placement-cost}
\end{equation}

We next construct stationary random centre sets $\mathcal A_k$.
In the discrete case let $(U_t)_{t\in\mathsf T}$ be iid uniform
random variables on $(0,1)$, put
$D_k=(K_k-K_k)+(K_k-K_k)$, and set
\[
 \mathcal A_k=\{s\in\mathsf T:U_s<U_{s+d}\text{ for every }
   d\in D_k\setminus\{0\}\}.
\]
Their laws are stationary and
$\Pp\{s\in\mathcal A_k\}=|D_k|^{-1}=d_k^{-1}$.
Since $D_k$ is symmetric, two centres cannot differ by an element
of $D_k$: each would have to have the smaller uniform label.
Consequently, if $1\le n,j\le k$, at most one translate
$s+J_n$, $s\in\mathcal A_k$, can meet $K_j$.

On $\mathbb R^m$, take $V$ uniform on $(0,1)^m$, put
$p_k=8(k+1)$, and set $\mathcal A_k=p_k(V+\mathbb Z^m)$.
A uniform phase makes this set stationary under every translation,
with intensity $p_k^{-m}=d_k^{-1}$. If $1\le n,j\le k$, two
translates of $J_n=[-n-1,n+1]^m$ meeting $K_j=[-j,j]^m$ would have
centres at distance at most $2(n+j+1)<p_k$ in the sup norm,
which is impossible.
In both cases the centre sets satisfy
\begin{equation}
 \E\Bigl\{\sum_{s\in\mathcal A_k}g(s)\Bigr\}
 =\frac1{d_k}\int_{\mathsf T}g(s)\,\lambda(ds),\qquad g\ge0.
 \label{sr:reward}
\end{equation}

Independently of these auxiliary random variables, draw one mark
$\zeta$ with probability law $K^{-1}d_{k(z)}\mathsf M(dz)$.
With $a_*=K^{1/\beta}$, define
\[
 X(t)=a_*\sum_{s\in\mathcal A_{k(\zeta)}}P_\zeta(t-s).
\]
The supports are disjoint and locally finite, so $X$ is a measurable
$\C$-valued random field, continuous when $\mathsf T=\mathbb R^m$.
Conditional on $\zeta$, its law is stationary; hence so is its
unconditional law. Nonnegativity is preserved.
 
Next, fix the test $\Phi$ from \eqref{sr:raw-limit}. Conditional on
$\zeta=z$, if $k(z)\ge j$, at most one path meets $K_j$, and
\[
 \Phi(u^{-1}\od X)
 =\sum_{s\in\mathcal A_{k(z)}}
       \Phi((u/a_*)^{-1}\od B^sP_z).
\]
If $k(z)<j$, then $S_z\le s_{j-1}$, and both sides vanish once
$\min_{1\le i\le q} u_i>Na_*s_{j-1}$; for $j=1$ this case is empty.
Thus the identity holds for every mark at all sufficiently large
thresholds. Integrating the mark law and using \eqref{sr:reward}
cancels the factor $d_{k(z)}$, giving
\begin{equation}
 \chi_\alpha(u)\E\{ \Phi(u^{-1}\od X)\}
 =\frac{a_*^\beta}{K}\lambda_{u/a_*}(\Phi)
 =\lambda_{u/a_*}(\Phi)\longrightarrow\nu(\Phi), \qquad u\to \infty.
 \label{sr:transfer}
\end{equation}
Absolute finiteness follows from the bounds preceding
\eqref{sr:raw-limit}. For each $N$, every compact restriction of
$\1_{O_N}\mu_u^X$ therefore converges weakly to the corresponding
restriction of $\1_{O_N}\nu$ as $u\to\infty$, including total masses.
Fix any sequence $(u^{(k)})_{k\ge1}$ diverging componentwise as $k\to\infty$ and put
$\eta_k=\1_{O_N}\mu_{u^{(k)}}^X$ and $\eta=\1_{O_N}\nu$.
For each $j\ge N$, the convergence
$(\pi_j)_\#\eta_k\Longrightarrow(\pi_j)_\#\eta$ as $k\to\infty$ gives uniform
tightness of these finite measures. Thus, for $\varepsilon>0$,
choose compact sets $C_j\subset C(K_j,\mathbb R^q)$ such that
\[
 \sup_{k\ge1}\eta_k\{f:\pi_jf\notin C_j\}
 \le\varepsilon2^{-(j-N+1)},\qquad j\ge N.
\]
The set $C=\bigcap_{j=N}^\infty\pi_j^{-1}(C_j)$ is compact:
diagonal extraction gives convergent restrictions on every $K_j$;
the limits agree on overlaps and define a path in $\C$.
Moreover,
\[
 \sup_{k\ge1}\eta_k(C^c)
 \le\sum_{j=N}^\infty\sup_{k\ge1}\eta_k\{f:\pi_jf\notin C_j\}
 \le\varepsilon.
\]
Every subsequential weak limit has compact restrictions
$(\pi_j)_\#\eta$, $j\ge N$, which determine $\eta$. Hence
$\eta_k\Longrightarrow\eta$ as $k\to\infty$, and arbitrariness of $(u^{(k)})$ yields
$\1_{O_N}\mu_u^X\Longrightarrow\1_{O_N}\nu$ as $u\to\infty$.
Lemma~\ref{133} completes the proof.
\end{proof}

\subsection{Finite moving-average calculations}
\label{230}

\begin{proof}[Proof of the finite-filter claims]
The construction commutes with shifts and is therefore strictly stationary.
We verify MrRV using the relative anchors \((0,\ell)\), as in
Corollary~\ref{50}.

Condition on a volatility path $v$ with $0<v(s)<\infty$ for every
$s\in\mathbb Z$ as holds almost surely. The rows are then independent.
For a finite set $J\subset\mathbb Z$ put
$S_J=\bigcup_{j=0}^{\ell_0-1}(J-j)$ and
$W_t(v)=\sum_{j=0}^{\ell_0-1}v(t-j)^\alpha$. The Pareto bound
$\Pp\{\eta>z\}\le z^{-\alpha}$ for all $z>0$ implies
\begin{equation}
 u^\alpha\Pp\left\{\max_{t\in J}X_i(t)>yu\mid V=v\right\}
 \le \ell_0^\alpha y^{-\alpha}\sum_{s\in S_J}v(s)^\alpha,
 \qquad u,y>0.
 \label{156}
\end{equation}
The bound is uniform in $u$ and its right-hand side is a finite
function of the observed volatility coordinates.

Fix a nonempty finite $J$, put $d_J=|S_J|$ and
$B_v=\max_{s\in S_J}v(s)$, and write $\E_v$ for expectation
conditional on $V=v$. Let $F$ be a bounded continuous function of
the row vector on $J$, vanishing when its maximum norm is at most
$\varepsilon>0$. For $M>\varepsilon$, choose a continuous cutoff
$F_M$ which agrees with $F$ on the ball of radius $M$ and vanishes
outside the ball of radius $2M$, with $|F_M|\le|F|$. Denote its
modulus of continuity by $\omega_M$.

For fixed $0<\delta<\varepsilon/(2d_JB_v)$, the probability that
two raw innovations \(\eta_{i,s}\), \(s\in S_J\), exceed \(\delta u\)
is at most
$\binom{d_J}{2}(\delta u)^{-2\alpha}$ for all sufficiently large $u$.
If none exceeds $\delta u$, then
$\|u^{-1}X_i|_J\|_\infty\le d_JB_v\delta<\varepsilon/2$,
so all test values vanish. If exactly one does, the other sources
perturb the normalised row vector by at most
$d_JB_v\delta$. A nonzero difference of test values then requires the
large-source vector to have norm at least $\varepsilon/2$; the scaled
probability of this event is at most
$d_J(2B_v/\varepsilon)^\alpha$, independently of $\delta$.
Consequently, we have 
\begin{equation}
 \begin{aligned}
 &\limsup_{u\to\infty}u^\alpha
   \Bigl|\E_v\{F_M(u^{-1}X_i|_J)\} -\sum_{s\in S_J}
   \E\{F_M(u^{-1}v(s)\eta_{i,s}\varphi_s|_J)\}\Bigr|\\
 &\quad\le d_J(2B_v/\varepsilon)^\alpha\omega_M(d_JB_v\delta).
 \end{aligned}
 \label{189}
\end{equation}
Letting $\delta\downarrow0$ removes this error. In each source term,
the scaled Pareto density is
$\1_{\{r>1/u\}}\alpha r^{-\alpha-1}\,dr$. Radial integration and
then removal of the upper cutoff give the conditional scalar tail
measure
\begin{equation}
 \nu_v^{\mathrm{sc}}(F)=\sum_{s\in\mathbb Z}\int_0^\infty
       F(rv(s)\varphi_s)\alpha r^{-\alpha-1}\,dr.
 \label{157}
\end{equation}
The error from replacing $F$ by $F_M$ is $O(M^{-\alpha})$,
uniformly after multiplication by $u^\alpha$, by \eqref{156};
the individual source terms satisfy the same bound by their Pareto
tails. Thus letting $M\to\infty$ proves finite-dimensional convergence as $u\to\infty$.

To obtain whole-path convergence fix $h\in\mathbb Z$ and $a>0$ and
consider the finite measures
\[
 \mu_{u,v}^{h,a}(B)=u^\alpha\Pp\{u^{-1}X_i\in B,\ X_i(h)>au                               \mid V=v\}.
\]
The limiting measure is $\1_{\{f(h)>a\}}\nu_v^{\mathrm{sc}}$ with total
mass $a^{-\alpha}W_h(v)$. Its anchor boundary has zero mass by the
radial density, so the finite-dimensional convergence above also proves
convergence of the restricted finite-dimensional measures and their
masses. Lemma~\ref{135}, applied to these finite measures with
$q=1$, gives weak convergence on the whole scalar path space
as $u\to\infty$.

For a relative anchor $(0,\ell)$ and levels $a_1,a_2>0$ conditional independence
identifies the normalised joint law on the anchor event with
$\mu_{u_1,v}^{0,a_1}\otimes\mu_{u_2,v}^{\ell,a_2}$. Both finite
measures converge whenever $u=(u_1,u_2)\to\infty$ without a restriction on
their relative rates. Their product masses are bounded by
\[
 \ell_0^{2\alpha}(a_1a_2)^{-\alpha}W_0(V)W_\ell(V),
\]
which is integrable because
$\E\{W_h(V)^2\}\le \ell_0^2\E\{V(0)^{2\alpha}\}<\infty$.

For every bounded continuous test function on the product path space,
the absolute value of its integral against the conditional product
measure is bounded by its supremum norm times this integrable bound,
independently of the thresholds. Dominated convergence therefore
permits averaging the conditional weak limits along every
componentwise diverging threshold sequence; uniform convergence
in the volatility path is not required.

In each source pair the substitutions $r_iv(s_i)\mapsto r_i$
and $\delta=s_2-s_1$ yield
\eqref{89} by stationarity. At anchors
$(0,\ell)$ write the source times as $-a$ and $\ell-b$ with
$0\le a,b<\ell_0$. Their relative lag is $\ell+a-b$, so summing these
source pairs proves \eqref{91}.
Only finitely many source pairs meet a given product window and
$w_\delta\le\E\{V(0)^{2\alpha}\}$, so the target is boundedly
finite and nonzero. Its radial densities give multihomogeneity, and
its source pulses lie in \(\C_\Pi^\circ\). Formula~\eqref{89} also
shows shift invariance, so Corollary~\ref{50} proves MrRV\@. Averaging
\eqref{157} also proves
\eqref{90} because $\E\{V(s)^\alpha\}=1$.

For the two models in \eqref{98}, write
$V(t)^\alpha=1+c\sum_{j=0}^3 b_j\xi_{t-j}$, with $b_j=0$ outside
$\{0,1,2,3\}$. Positivity follows from $7c<1$, and boundedness,
stationarity and centring imply \eqref{86}. Independence and centring
also imply
\[
 w_k=1+c^2\sum_{j\in\mathbb Z}b_jb_{j+k},\qquad k\in\mathbb Z.
\]
For $(b_0,b_1,b_2,b_3)=(0,2,1,4)$ and $(1,0,4,2)$, the squared
coefficient sum is $21$ and the lag products at $k=1,2,3$ are,
respectively, $(6,8,0)$ and $(8,4,2)$. They vanish for $|k|\ge4$,
and $w_{-k}=w_k$. Both lag-product triples have sum $14$ and
weighted sum $1\cdot C_1+2\cdot C_2+3\cdot C_3=22$, where
$C_k=\sum_{j\in\mathbb Z}b_jb_{j+k}$. Consequently, for every $m\ge3$
\[
 \begin{aligned}
 M_2^A(m)=M_2^B(m)
 &=m^2+c^2\left(21m+2\sum_{k=1}^3(m-k)C_k\right)\\
 &=m^2+(49m-44)c^2.
 \end{aligned}
\]
Taking $m=\ell_0$ and $m=\ell_0-1$ proves the common-time anchor
identity in \eqref{99} and  by
Remark~\ref{rem:finite-filter-common-time} the stated equality of
$H_2^A(T)$ and $H_2^B(T)$ for every $T\in\mathbb N$.
For the relative lag two, put
$\Delta_k=(w_k^A-w_k^B)/c^2$. The preceding lag products give
$\Delta_{\pm1}=-2$, $\Delta_{\pm2}=4$, $\Delta_{\pm3}=-2$,
and $\Delta_k=0$ otherwise. Thus \eqref{91} yields
\[
 p_{(0,2)}^A-p_{(0,2)}^B
 =c^2\sum_{k=-3}^3(\ell_0-|k-2|)_+\Delta_k,
\]
which equals $2c^2$ for $\ell_0=4$ and $4c^2$ for $\ell_0\ge5$,
as asserted in \eqref{99}.
Finally, the coefficient vectors have the same nonzero entries, so
independence and identical distribution of the Rademacher variables imply  the same one-time volatility law. Formula~\eqref{90} establishes  the
claimed equality of the scalar path-tail measures.
\end{proof}



\bibliographystyle{ieeetr}
\bibliography{EEEA_RV_V55}
\end{document}